\documentclass[12pt,reqno]{amsart}
\usepackage{amsthm}
\usepackage{cases}
\usepackage{amscd}
\usepackage{amsfonts}

\usepackage{amssymb}
\usepackage{amsmath}
\usepackage[dvipdfmx]{graphicx}
\usepackage{booktabs}

\usepackage{subfigure}
\usepackage{algorithm}
\usepackage{algpseudocode} 
\renewcommand{\algorithmicrequire}

\usepackage[colorlinks=true, linkcolor=blue, citecolor=blue, urlcolor=blue]{hyperref}
\usepackage{cleveref}
\usepackage{color}
\usepackage{datetime}

\allowdisplaybreaks

\makeatletter
\newcommand\figcaption{\def\@captype{figure}\caption}
\newcommand\tabcaption{\def\@captype{table}\caption}
\makeatother

\usepackage{geometry}
\usepackage{algorithm}
\usepackage{multirow}
\usepackage{mathrsfs}

\newtheorem{remark}{Remark}
\newtheorem{theorem}{Theorem}

\newtheorem{example}{Example}

\newcommand{\beq}{\begin{equation}}
\newcommand{\eeq}{\end{equation}}
\newcommand{\bea}{\begin{eqnarray}}
\newcommand{\eea}{\end{eqnarray}}
\newcommand{\beas}{\begin{eqnarray*}}
\newcommand{\eeas}{\end{eqnarray*}}

\newcommand{\secref}[1]{\hyperref[#1]{Section~\ref*{#1}}}
\newcommand{\appref}[1]{\hyperref[#1]{Appendix~\ref*{#1}}}

\begin{document}
\title[Operator-Learning-Driven Method for Inverse Source Problems]{Transcending Sparse Measurement Limits: Operator-Learning-Driven Data Super-Resolution for Inverse Source Problem}

\author[G. Pan, J. Zhou, X. Liu, Y. Huang and N. Yi]{
    Guanyu Pan$^{1}$, 
    Jianing Zhou$^{1}$, 
    Xiaotong Liu$^{3}$,
    Yunqing Huang$^{1, 2, * \footnote{Corresponding author.}}$, 
    and Nianyu Yi$^{1, 2}$
}

\address{$^{1}$ School of Mathematics and Computational Science, Xiangtan University, Xiangtan 411105, China}
\address{$^{2}$ National Center for Applied Mathematics in Hunan, Xiangtan 411105, China}
\address{$^{3}$ College of Mathematics and Informatics, South China Agricultural University, Guangzhou 510006, China}
\email{202421511213@smail.xtu.edu.cn (G. Pan), 202421511272@smail.xtu.edu.cn (J. Zhou), \newline
xiaotongliu163@163.com (X. Liu), huangyq@xtu.edu.cn (Y. Huang), yinianyu@xtu.edu.cn (N. Yi)}



\begin{abstract}
Inverse source localization from Helmholtz boundary data collected over a narrow aperture is highly ill-posed and severely undersampled, thereby undermining classical solvers (e.g., Direct Sampling Method). We present a modular framework that significantly enhances multi-source localization even with extremely sparse single-frequency measurements.
First, we extend a uniqueness theorem for the inverse source problem, proving that a unique solution is guaranteed despite limited viewing apertures.
Second, we employ a  Deep Operator Network (DeepONet) with a branch–trunk architecture to interpolate the sparse measurements, lifting six to ten samples within the narrow aperture to a sufficiently dense synthetic aperture.
Third, the super-resolved field is fed into the Direct Sampling Method (DSM). For a single source, we derive an error estimate showing that sparse data alone can achieve grid-level precision. In two- and three-source trials, localization from raw sparse measurements is unreliable, whereas DeepONet-reconstructed data reduce localization error by about an order of magnitude and remain effective with apertures as small as $\frac{\pi}{4}$.
By decoupling data interpolation from inversion, the framework allows the interpolation and inversion modules to be swapped with  neural operators and classical algorithms,  respectively, providing a practical and flexible design that improves localization accuracy compared with standard baselines.
\end{abstract}

\thanks{$^*$ Corresponding author.}

\maketitle


\section{Introduction}

The inverse source problem (ISP) is fundamental to acoustics, electromagnetism and biomedical imaging, where the goal is to recover the locations and strengths of unknown sources from the boundary measurements \cite{anastasio2007application, arridge1999optical, bao2002inverse}.
When dense, full-aperture data are available over multiple wavenumbers, a range of classical algorithms, including Newton-type iterative schemes, variational optimization, Bayesian inversion, recursive linearization, and direct sampling method (DSM) can achieve high spatial resolution \cite{bao2015inverse, colton1996simple, el2011inverse, hohage1998convergence, ito2012direct, kirsch1987optimization,  kirsch1998characterization, li2021quality,  YangZhangBoandZhang}.  
In practice, however, experimental constraints are far more stringent: deep-ocean acoustics and through-the-wall radar imaging often yield only a handful of low-frequency samples collected over a narrow aperture \cite{ahmad2008multi, jensen2011computational}. 
Such extreme sparsity renders the inverse problem highly ill-posed, expands the null space of the forward operator, and equivalently, enlarges the admissible solution set, leading to non-uniqueness and a pronounced loss of resolution \cite{bleistein1977nonuniqueness, devaney2003nonuniqueness, liu2022deterministic}. 
Nonetheless, recent sparse-data inversion theory shows that meaningful recovery is still possible with very few measurements \cite{ji2020identification, ji2021source}. 
Designing strategies that allow classical solvers to retain-or even enhance-their resolving power under severe limitations of data and aperture is therefore a critical open problem, and it is precisely the focus of the present work.
 
We reformulate the challenge as learning an interpolation operator that converts sparse measurements into a sufficiently dense data, thereby enabling standard inversion algorithms to perform effectively. Recent advances in neural-operator learning \cite{cybenko1989approximation,kovachki2023neural,lu2021learning} demonstrate that neural operator networks can learn a continuation operator, which maps a limited set of measurements to an acceptable field distribution over an aperture. 
In our framework, such a network is trained on a large ensemble of synthetic source configurations that cover the anticipated frequency range. 
During inference, the network augments any new experiment with virtual measurements that are mutually consistent and encode statistical structure extracted from the training ensemble, without altering the existing physical model.

The reconstructed field is then fed into DSM, which acts as a fast, mesh-free solver and, when desired, furnishes a high-quality initial guess for subsequent gradient-based or Bayesian refinements \cite{li2020extended, li2021quality,Qin_2023}. 
Compared with fully end-to-end approaches, this hybrid strategy preserves the interpretability and reliability of classical inversion algorithms while substantially improving accuracy relative to reconstructions obtained directly from the sparse measurements. 

In addition, we further extend the uniqueness theorem of Bao \emph{et al.} \cite{bao2010multi} for the multi-frequency ISP to finite measurable apertures. Specifically, we prove that uniqueness still holds for the Dirac-source ISP when only a finite measurable aperture is available, thereby providing a rigorous theoretical foundation for inversion algorithms that rely on partial measurement data.
For the single-source case, we derive a preliminary error estimate for the DSM, which corroborates the method’s theoretical validity and helps explain the robustness observed in our numerical experiments.

Our main contributions can be summarized as follows: 
\begin{itemize}
    \item[i)] A modular framework is proposed that incorporates a neural operator interpolator. This framework enables the transformation of sparse boundary data into a dense synthetic aperture dataset, thereby substantially augmenting the reconstruction accuracy of the DSM.

    \item[ii)] A finite-aperture uniqueness theorem for Dirac-source inverse scattering is established, which provides a rigorous theoretical underpinning for limited-view reconstruction methodologies.
      
    \item[iii)] Both a prior and a posterior error estimations for the DSM are derived in the single-source regime. These estimations not only inform the selection of sampling parameters, but also elucidate the robustness of the proposed method.
\end{itemize}  

All codes, trained networks and datasets used in this study will be released publicly upon publication, enabling complete reproduction of our results. 

The rest of the paper is organized as follows.
\secref{sec2} develops our methodology. 
In \secref{sec2.1} we formalize the inverse source problem and establish a finite-aperture uniqueness theorem for Dirac sources. 
\secref{sec2.2} revisits the direct sampling method,  derives a priori and a posteriori error estimates for the single-source case, and discusses limitations under sparse measurements. 
\secref{sec2.3} introduces the operator-learning interpolation module and the hybrid framework that couples interpolation with DSM. 
\secref{sec3} reports numerical experiments. 
\secref{sec3.1}validates two-source localization and \secref{sec3.2} extends to three sources, comparing DSM with and without operator-based interpolation. 
\secref{sec4} concludes with a discussion of implications and future directions. 
\appref{apd:1} collects technical estimates used in the DSM error analysis.

\section{Methodology}\label{sec2}

\subsection{Inverse source problem}\label{sec2.1}

Given a source distribution $F \in L^{2}(\mathbb{R}^{2})$ with $\text{supp} \, F \subset V$,  where $V$ is a bounded domain in $\mathbb{R}^{2}$. 
For solving the acoustic field $u \in H_{loc}^{1}(\mathbb{R}^{2})$, the equations it obeys are 
\begin{subequations}  
    \begin{align}  
        \Delta u({x}) + k^{2} u({x}) &= F({x}), \quad {x} \in \mathbb{R}^{2}, \label{eq:helmholtz_a} \\
        \lim_{r\to\infty} \sqrt{r} \left(\frac{\partial u}{\partial r} ({x}) - iku({x}) \right) &= 0, \quad r = |{x}|, \label{eq:helmholtz_b}  
    \end{align}  
\end{subequations}
where \eqref{eq:helmholtz_b} is the Sommerfeld radiation condition that guarantees the uniqueness of the solution $u$. 
The fundamental solution to the Helmholtz equation is given by :  
\begin{equation*}
    \Phi_{k}({x}, {y}) = \frac{i}{4}H_{0}^{(1)}(k|{x} - {y}|),  
\end{equation*}
where $H_{0}^{(1)}$ denotes the Hankel function of the zeroth order and the first kind. 
As is well-known, $\Phi_{k}({x}, {y})$ satisfies:  
\begin{equation}\label{Phik-Dirac}  
    \Delta \Phi_{k}({x}, {y}) + k^{2} \Phi_{k}({x}, {y}) = -\delta(|{x} - {y}|),
\end{equation} 
where $\delta$ is the Dirac distribution. 
Given these conditions, the solution $u$ to \eqref{eq:helmholtz_a} and \eqref{eq:helmholtz_b} can be represented as:  
\begin{equation}
    \label{eq:general_solution}
    u({x}, k) = \int_{\mathbb{R}^2} \Phi_{k}({x}, {y}) F({y}) \, \mathrm{d}y. 
\end{equation}  

In this paper, we focus on tackling the inverse source problem corresponding to the forward problem described by \eqref{eq:helmholtz_a} and \eqref{eq:helmholtz_b}, formulated as:  
\begin{equation}\label{eq:general_helmholtz}
\left\{\begin{aligned}
        \Delta u\left( {x} \right) + k^2 u\left( {x} \right) = &f\left( {x} \right), \qquad {x} \in \Omega, \\
        \mathfrak{B} \left( u\left( {x} \right) \right) = &g\left( {x} \right), \qquad  {x} \in \Gamma.
\end{aligned}
\right.
\end{equation}
Here, $k$ is the wavenumber, $\Omega \subset \mathbb{R}^2$ has a piecewise Lipchitz boundary $\Gamma$ satisfying the interior cone condition. 
The operator $\mathfrak{B}$ denotes a boundary condition applied to $\Gamma$ (e.g., Dirichlet, Neumann, or Robin condition). The source $f(x)$ is Dirac-source, i.e., 
\begin{equation}\label{eq:source_f}
f(x)=\sum^M_{m=1}\lambda_m\delta(|x-z_m|),\end{equation}
where $z_m (m=1,\cdots,M)$ are point-like scatterers, and $\lambda_m\ne 0 ( m=1,\cdots,M)$ are  scattering
strengths.
The inverse source problem can  be expressed as follows: Given the boundary measurements $g\left( {x_i} \right)$ at $x_i\in  \Gamma, i = 1, \cdots, N$, our objective is to determine the source $f({x})$ of the form \eqref{eq:source_f}. 

For $\rho>0$ and $x\in \mathbb{R}^2,$ let $B_{\rho}(x)$ denote the open disk with radius $\rho$ and center $x$, i.e. 
\begin{equation*}
    B_{\rho}(x)=\{y\in \mathbb{R}^2,|x-y|<\rho\}.
\end{equation*}
In practical inversion, considering sources with \( \mathop{\mathrm{supp}} F \subset B_{R_0} \) and \( F \in L^1(B_{R_0}) \) ,  we define the scattering operator \cite{bao2010multi}
\begin{equation*}
    \mathfrak{L}_k  \left( F \right)(x) :=\int_{B_{R_0}}{\Phi _k\left( {x},{y} \right)}F\left( {y} \right) \mathrm{d}{y}.
\end{equation*}
The ISP at a fixed frequency may be established as follows: find \textit{F} satisfies the linear equations. 
\begin{equation*}
    \mathfrak{L}_k \left( F \right)(x) =\psi(x,k), \quad x\in \Gamma,
\end{equation*}
where $\psi(x,k)$ is defined as $u(x,k)$  on the measurement curve $\Gamma$.

According to \cite[Theorem 4.2]{D.CandR.K}, due to the infinite dimension of $ L^1(B_{R_0})$, the equation $\mathfrak{L}_k \left( F \right)(x) =\psi(x,k)$ is ill-posed. 
In \cite[Theorem 3.1-3.3]{bao2010multi}, Bao \emph{et al.} study the multi-frequency inverse problem and prove the uniqueness and stability estimates under the assumption that the source term is $L^2$-integrable and full-aperture measurements.
In this paper, we depart from the framework of prior works and adopt a complex analytical approach to establish the uniqueness result for acoustic scattering problems with Dirac-source terms.

\begin{theorem}\label{thm:uniqueness}
Assume that the measurement curve $\Gamma$ satisfies $\Gamma \cap V =\emptyset$, given $u|_{\Gamma}=\psi$, for the Helmholtz equation with Dirac source terms
\begin{equation}
\Delta u + k^{2}u=\sum_{j = 1}^{m}a_{j}\delta(|x - z_{j}|),
\end{equation}
the source term is uniquely determined.
\end{theorem}

\begin{proof}
Assume that 
\[\displaystyle f_{1}(x)=\sum_{q = 1}^{m}\alpha_{q}\delta(|x - x_{q}|),\qquad \alpha_q\ne 0,\quad  q=1,\cdots,m,\]
and 
\[
f_{2}(x)=\sum_{j = 1}^{n}\beta_{j}\delta(|x - y_{j}|),\qquad \beta_j\ne 0,\quad j=1,\cdots,n,
\]
are two groups of source terms. From equation \eqref{Phik-Dirac} and the superposition principle, we know that their corresponding solutions are, respectively, 
\begin{equation}
    u_{1}(x)=-\sum_{q = 1}^{m}\alpha_{q}\Phi_{k}(x,x_{q}),\quad u_{2}(x)=-\sum_{j = 1}^{n}\beta_{j}\Phi_{k}(x,y_{j}), 
\end{equation}
where $\Phi_{k}(x,x_{q})=\frac{i}{4}H_{0}^{(1)}(k|x - x_{q}|)$.

For the partial aperture measurement curve $\Gamma$ ( $\Gamma\cap V = \emptyset$), suppose that source terms $f_1$ and $f_2$ produce the same measurement data in $\Gamma$, i.e., $u_{1}$ and $u_{2}$ satisfy the conditions 
\begin{equation}
    \psi_{1}|_{\Gamma}=\psi_{2}|_{\Gamma}.
\end{equation}
Let $u=\psi_{1}-\psi_{2}$, then $\varphi$ satisfies the following Helmholtz equation
\begin{equation}\label{imhomogeneous}
\Delta u + k^{2}u=f_{1}-f_{2}.
\end{equation}
When we consider the region outside  $V$,  then the equation \eqref{imhomogeneous} is a homogeneous Helmholtz equation
\begin{equation}\label{homogeneous Helmholtz equation}
\Delta u + k^{2}u = 0,\qquad  x\in \mathbb{R}^2/V.
\end{equation}
Notice that $u$ can be expressed as
\begin{equation}\label{varphi function} 
u=u_{1}-u_{2}=\sum_{j = 1}^{n}\beta_{j}\Phi_{k}(x,y_{j})-\sum_{q = 1}^{m}\alpha_{q}\Phi_{k}(x,x_{q}).
\end{equation}
We can find a bounded domain $
\tilde{\Omega}\supset\Gamma, \tilde{\Omega}\supset V$, and set  
\[
\Upsilon 
= \tilde{\Omega} \setminus\{x_{1},\cdots,x_{m},y_{1},\cdots,y_{n}\}.\]
It is easily to check that $u$ is an analytic function on $\Upsilon$. Since $u\equiv0$ on the curve $\Gamma$, there exists an accumulation point on $\Gamma$. 
From unique continuation \cite{LHormander}, we deduce that $u\equiv0$ in $\Upsilon $, and the following formula holds 
\begin{equation}\label{2.11}
\sum_{j = 1}^{n}\beta_{j}\Phi_{k}(x,y_{j})-\sum_{q = 1}^{m}\alpha_{q}\Phi_{k}(x,x_{q}) = 0,\quad x\in \Upsilon.
\end{equation}
We consider the cases $x\rightarrow y_{s},s=1,\cdots,n$. 
To ensure that the left-hand side of \eqref{2.11} is zero, we must have $m = n$ and $x_{q}=y_{s}$; otherwise, there exists $q$ such that $x_q\ne y_s$ for $s=1,\cdots,n $, and  
\begin{equation}\label{2.12}
\sum_{j=1}^n{\beta _j}\Phi _k(x,y_j)-\sum_{q=1}^m{\alpha _q}\Phi _k(x,x_q)\rightarrow \infty ,\quad x\rightarrow y_s,
\end{equation}
a contradiction will be derived. 

Taking $\Delta+k^2$  on both sides of the equation\eqref{2.11}, and then we multiply this resulting expression by a test function $\phi_\ell \in C^0(\tilde{\Omega})$ and then integrate over the domain  $\tilde{\Omega}$, which yields: 
\begin{equation}\begin{aligned}
    \int_{\tilde{\Omega}}&{\left( \sum_{j=1}^n{\left( \beta _j-\alpha _j \right) \left( \Delta +k^2 \right) \Phi _k(x,y_j)} \right)}\phi_\ell \left( x \right) \mathrm{d}x\\
    =&\sum_{j=1}^n{\left( \alpha _j-\beta _j \right) \phi_\ell \left( y_j \right)}=0,\qquad \ell=1, \cdots, n.
\end{aligned}
\end{equation}
By setting   $\phi_\ell(y_j)=\delta_{\ell j},$ we can deduce that 
$\alpha_j=\beta_j,j=1,\cdots,n$.
\end{proof}

\subsection{Direct sampling method}\label{sec2.2}

Given the measured near-field data, Li \emph{et al.} propose a DSM to reconstruct the locations of point-like  sources \cite{li2021quality}. The DSM proposed is to reconstruct point-like scatterers using near-field data from all directions by a single incident plane wave in \cite{ito2012direct}. 
Now, for the readers' convenience, we recall the DSM with partial data following \cite{li2021quality}. Since the DSM-based derivation below assumes a simple closed measurement curve, we introduce a new set of symbols to distinguish this setting from the notation used previously.

Let $\tilde{V}$ be the sampling domain such that $\tilde{V} \subsetneq \tilde{D}$ and $n$ be the unit outward normal to $\partial \tilde{D}$. 
Denote by $\mathcal{A} :=\left\{ k_i \right\} _{i=1}^{N}$ a finite of  wavenumbers, $\tilde{\Gamma}=\partial \tilde{D} $ is a closed curve external to $\tilde{V}$, and $\tilde{\Gamma} \cap \tilde{V}=\emptyset$, we denote their distance as $\text{dist}(\tilde{\Gamma} ,\tilde{V}):= \inf_{x\in \tilde{\Gamma},z\in \tilde{V}} |x-z|$.
For near-field data,  we first define the near-field function
\begin{align}
    I\left( z_p \right) &=\sum_{k_i\in \mathcal{A}}^{}{\int_{\tilde{\Gamma}}{\int_{\tilde{V}}{\Phi _k\left( x,y \right) F\left( y \right) \mathrm{d}y}}}\bar{\Phi}_k\left( x,z_p \right) \mathrm{d}s\left( x \right) \nonumber
    \\
    &=\sum_{k_i\in \mathcal{A}}^{}{\int_{\tilde{V}}{\int_{\tilde{\Gamma}}{\Phi _k\left( x,y \right)}\bar{\Phi}_k\left( x,z_p \right) \mathrm{d}s\left( x \right) F\left( y \right) \mathrm{d}y}},\label{near-field function}
\end{align}
due to the equation
\eqref{Phik-Dirac}, it can be obtain that
\begin{align*}
    \int_{\tilde{D}}{\left( \Delta \Phi _{k}\left( x,y \right) +k^2\Phi _{k}\left( x,y \right) \right)}\bar{\Phi}_{k}\left( x,z_p \right) \mathrm{d}x=-\bar{\Phi}_{k}\left( y,z_p \right) ,
    \\
    \int_{\tilde{D}}{\left( \Delta \bar{\Phi}_{k}\left( x,z_p \right) +k^2\bar{\Phi}_{k}\left( x,z_p \right) \right)}\Phi _{k}\left( x,y \right) \mathrm{d}x=-\Phi _{k}\left( y,z_p \right).
\end{align*}
Using Green's formula and the Sommerfeld radiation condition, we derive 
\begin{align*}
    \Phi _k\left( y,z_p \right) -\bar{\Phi}_k\left( y,z_p \right) 
    &=\int_{\tilde{\Gamma}}{\left\{ \bar{\Phi}_{k}\left( x,z_p \right) \frac{\partial \Phi _{k}\left( x,y \right)}{\partial n}-\Phi _{k}\left( x,y \right) \frac{\partial \bar{\Phi}_{k}\left( x,z_p \right)}{\partial n} \right\}}\mathrm{d}s\left( x \right) 
    \\
    &=2ik\int_{\tilde{\Gamma}}{\bar{\Phi}_k\left( x,z_p \right) \Phi _k\left( x,y \right)}\mathrm{d}s\left( x \right) \\
    & \qquad+o\left( |x|^{-\frac{1}{2}} \right) \left( \int_{\tilde{\Gamma}}{\left\{ \bar{\Phi}_k\left( x,z_p \right) -\Phi _k\left( x,y \right) \right\}}\mathrm{d}s\left( x \right) \right) 
\\
&=2ik\int_{\tilde{\Gamma}}{\bar{\Phi}_k\left( x,z_p \right) \Phi _k\left( x,y \right)}\mathrm{d}s\left( x \right) +o\left( |x|^{-\frac{1}{2}} \right) O\left( |x|^{-\frac{1}{2}} \right) O\left( |x| \right) 
\\
    &= 2ik\int_{\tilde{\Gamma}}{\bar{\Phi}_{k}\left( x,z_p \right) \Phi _{k}\left( x,y \right)}\mathrm{d}s\left( x \right)+o(1) \quad (|x| \rightarrow \infty).
\end{align*}
Denote 
\begin{equation*}
    \Im \left( \Phi _k\left( y,z_p \right) \right) = k\int_{\tilde{\Gamma}}{\bar{\Phi}_k\left( x,z_p \right) \Phi _k\left( x,y \right)}\mathrm{d}s\left( x \right)+o(1) \quad (|x| \rightarrow \infty),
\end{equation*}
where $\Im(\cdot)$ denotes the imaginary part. 
For $y,z_p\in \tilde{V}$, thus  we define the kernel function  about \eqref{near-field function}
\begin{equation*}
    H_k(y,z_p):=
    \int_{\tilde{\Gamma}}{\bar{\Phi}_k\left( x,z_p \right) \Phi _k\left( x,y \right)}\mathrm{d}s\left( x \right)= \frac{1}{4k}J_0\left( k|y-z_p| \right)+o(1), \quad (|x| \rightarrow \infty), k\in \mathcal{A},
\end{equation*}
where $J_0$ is the zeroth order Bessel function. 
From the asymptotic property of $J_0$, it has
\begin{align*}
    \underset{t\in \mathbb{R}}{\mathrm{argmax}}J_0\left( t \right) =0,  \qquad J_0\left( t \right) =\frac{\sin t+\cos t}{\sqrt{\pi t}}\left\{ 1+O\left( \frac{1}{t} \right) \right\} , \quad t\rightarrow \infty,
\end{align*}
this implies $I(z_p)$ decays $z_p \rightarrow \infty$, and  $I(z_p)$  leads to an indicator about the near-field data
\begin{align}\label{IDSM}
    I_{DSM}\left( z_p \right) =\frac{\left| \sum_{k_i\in \mathcal{A}}^{}{\langle u\left( x,k_i \right) ,\Phi _{k_i}\left( x,z_p \right) \rangle _{\mathrm{L}^2\left( \tilde{\Gamma} \right)}} \right|}{\sum_{k_i\in \mathcal{A}}^{}{\left\| u\left( x,k_i \right) \right\| _{\mathrm{L}^2\left( \tilde{\Gamma}\right)}\left\| \Phi _{k_i}\left( x,z_p \right) \right\| _{\mathrm{L}^2\left( \tilde{\Gamma} \right)}}},\qquad \forall z_p\in \tilde{V},
\end{align}
where the inner product is defined as 
\begin{equation*}
    \langle u\left( x,k_i \right) ,\Phi _{k_i}\left( x,z_p \right) \rangle _{\mathrm{L}^2\left( \tilde{\Gamma} \right)}=\int_{\tilde{\Gamma}}{u\left( x,k_i \right)}\bar{\Phi}_{k_i}\left( x,z_p \right) \mathrm{d}s\left( x \right) .
\end{equation*}
Finally we calculate $I_{DSM}(z_p), z_p\in \tilde{V}$, and find the local maximizers $z^{DSM}$ of $I_{DSM}(z_p)$, which provides the rough locations of the sources.

Subsequently, we conduct an error analysis of the DSM for a single-Dirac-source case $ f(x)=\lambda_1\delta(x-z_1), \lambda_1>0, z_1\in \tilde{V}$. 
\begin{theorem}\label{prior error estimation}
    For a specific frequency $k>0$ and single points $z_1$, if $ k\Theta \ge 15$, we have 
    \begin{align}
        |z_1-z^{DSM}|<\frac{1}{15k},
    \end{align}
    where $z^{DSM}$ is a  local maximizer of  $I_{DSM}(z_p)$,  $\Theta = \text{dist}(\tilde{\Gamma}_1,\tilde{V})$, and
    $\tilde{\Gamma}_1=\{(R\cos t, R\sin t)|t\in [-\theta/2,\theta/2] ,\theta \in (0, 2\pi) \}.$
\end{theorem}
\begin{proof}
First, we make a preliminary approximation for the formula \eqref{IDSM}
\begin{equation*}
    I_{DSM}\left( z_p \right) \approx \frac{\displaystyle\left| \frac{\lambda}{4k}J_0\left( k|z_1-z_p| \right) \right|}{\theta R\left(\displaystyle \frac{1}{M}\sum_{n=1}^M{|u\left( x_n,k \right) |^2} \right) ^{1/2}\left( \displaystyle\frac{1}{M}\sum_{n=1}^M{|\Phi _k\left( x_n,z_p \right) |^2} \right) ^{1/2}},
\end{equation*}
where the measurement locations $x_n\in \tilde{\Gamma}_1 ,(n=1,\cdots,M)$ are constrained by the \emph{far-field condition} $R>>d(\tilde{V})$, with the diameter of set $\tilde{V}$ defined as $d( \tilde{V} ) :=\underset{x_1,x_2\in \tilde{V}}{\mathrm{sup}}|x_1-x_2|$. 
Let $y=|z_1-z_p|,$ so it can be approximated 
\begin{align*}
    \theta R\left( \frac{1}{M}\sum_{n=1}^M{|u\left( x_n,k \right) |^2} \right) ^{1/2}I_{DSM}\left( z_p \right) 
    \approx & g\left( y \right) =\frac{\left| \frac{\lambda}{4k}J_0\left( ky \right) \right|}{|\Phi _k\left( ky+k\xi \right) |}\\
    =& \frac{\lambda}{k}\frac{\left| J_0\left( ky \right) \right|}{\sqrt{J_0\left( ky+k\xi \right) ^2+Y_0\left( ky+k\xi \right) ^2}},
\end{align*}
where $\xi =\underset{x_n\in \tilde{\Gamma}_1}{\mathrm{inf}}|x_n-z_1|\ge \Theta$. To find  the local maximizer of  $I_{DSM}$, we  differentiate  $g(y)$ and obtain
\begin{align}\label{Dg}
    g'\left( y \right) =&-\lambda J_1(ky)\frac{\text{sgn}\left( J_0(ky) \right)}{|H_{0}^{1}\left( ky+k\xi\right) |}+\lambda\frac{\text{sgn} \left( J_0(ky) \right) J_0(ky)}{|H_{0}^{1}\left( ky+k\xi \right) |^3}J_0\left( ky+k\xi \right) J_1\left( ky+k\xi \right) \nonumber \\ 
    &+\lambda\frac{\text{sgn} \left( J_0(ky) \right) J_0(ky)}{|H_{0}^{1}\left( ky+k\xi \right) |^3}Y_0\left(ky+k\xi \right) Y_1\left(ky+k\xi\right) .
\end{align}
Given $g'\left( 0 \right)>0, g'(1/(15k))<0, k\xi\ge k\Theta \ge 15$ (for the proof, see \appref{apd:1}), then $g'(y)=0$ has a root in  $(0, \frac{1}{15k})$. 
Therefore, we have completed the proof.
\end{proof}

\begin{remark}
    When $k\Theta\ge15$, the above theorem can provide a prior error estimation. In fact, after determining $\xi$, we set $x_0$ to satisfy the equation   $g'(x_0)=0$, since  $g(y)$ increases monotonically  in $(0, x_0)$ , which gives an a posterior error estimation about DSM in the single point source case.
\end{remark}

We use the following example to verify the validity of our error estimation.
\begin{example}\label{singleDiracradiator}
    The domain is fixed as $V=[-2,2]^2$. A single Dirac radiator is placed at $z_1=(1,0)$ with amplitude $\lambda=5$, and the measurement data are collected on the  finite-aperture  fan-shaped surface with a fixed radius
    \begin{equation*}
        \Gamma_n = \{ x|x=(R\cos \theta_n,R\sin \theta_n)  \}, \quad R=7,\theta_n\in[-\pi/n,\pi/n],\quad n=2,3,4,
    \end{equation*}
    with $M+1=51$ sensors,  they are uniformly distributed  along the arc of the fan-shaped aperture,  as defined by its coordinates:
    \begin{equation*}
        x_{n}^{\left( i \right)}=\left( R\cos \frac{2n\pi}{Mi},R\sin \frac{2n\pi}{Mi} \right) ,\qquad n=-\frac{M}{2},\cdots ,\frac{M}{2},\quad i=2,3,4.
    \end{equation*}
    Therefore, the smallest sensor-source separation is  
    \begin{equation*}
       \Theta_i=7-2\sqrt{2},\quad \xi_i=6,\quad i=2,3,4.
    \end{equation*}
    Synthetic measurements are generated at wavenumber $k=4$ and notice that  
    $k\Theta_i\approx 16.68\ge 15$, from  Theorem \ref{prior error estimation}, the prior error estimation is
        \begin{align}
        |z_1-z^{DSM}|<\frac{1}{60}\approx1.66\times10^{-2}.
    \end{align}
    In addition,  the a posteriori error of the DSM indicator  is calculated as $x_0=1.04\times10^{-3}$ via the bisection method.
    The reconstruction region is discretized into a uniform grid with spacing  $h=0.04$, and then $5\%$ Gaussian random noise is added to the measurements 
    \begin{equation*}
        \hat{u}\left(    \mathbf x^{\left( i \right)},k \right) :=u\left(    \mathbf x^{\left( i \right)},k \right) +Z_1,\quad \mathbf x^{(i)}=(x^{(i)}_{-M/2},\cdots,x^{(i)}_{M/2}),\qquad i=2,3,4,
    \end{equation*}
    where $Z_1\sim \mathcal{N}_{\mathbb{C}} \left( 0,0.05^2|u|^2 \right)$. \(\mathcal{N}_\mathbb{C}(\mu,\sigma^2)\) is a complex normal distribution with mean \(\mu\) and variance  \(\sigma^2\).
\end{example}

Figure \ref{fig:error_analysis_demo} displays the DSM indicator computed for each finite-aperture configuration. 
The indicator attains its global peak at $\boldsymbol z^{{DSM}}_n=(1,0),n=2,3,4$, recovering the true source location  within the grid resolution, this result indicates that   the measurement angle doesn't degrade the localization accuracy of the DSM.
We note that if the extreme point exceeds the grid resolution  ($\delta>h$),  the exact source location may not be  resolved. 
The single-source case confirms that the direct sampling method is both robust and effective . Furthermore, we can also observe that the number of sensors   $M\ge 5$ does not affect the positioning effect. 

\begin{figure}[htb]
    \centering
    \includegraphics[width=0.95\textwidth]{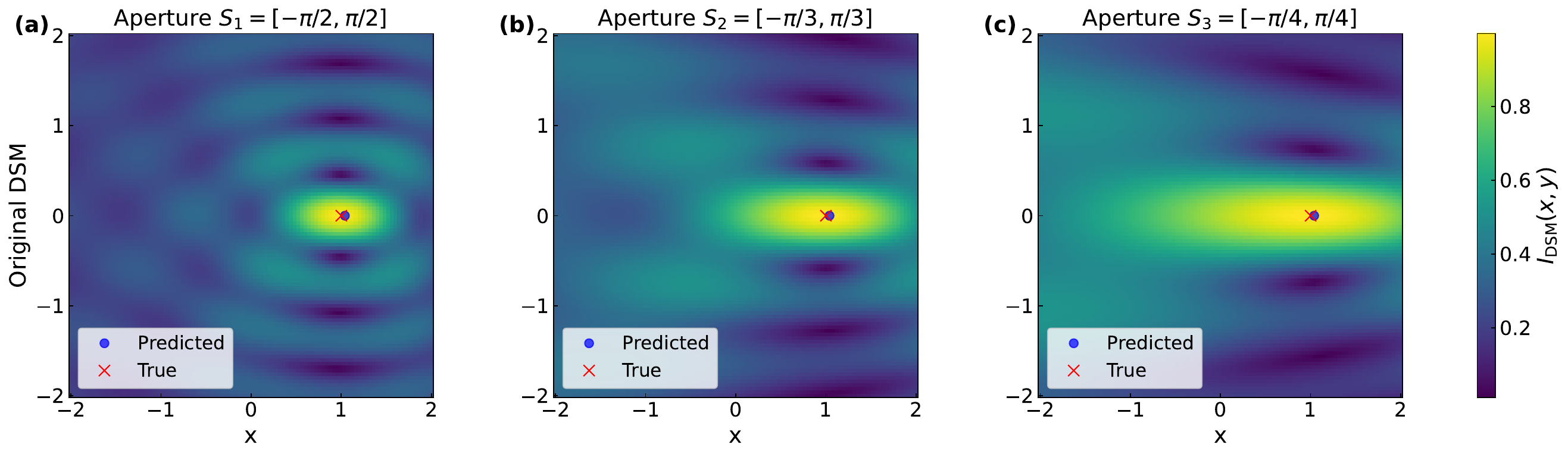}
    \caption{
        \textbf{DSM indicator for a single source (Example \ref{singleDiracradiator}).}
        Maps of $I_{DSM}(x)$ for apertures 
        (a) $S_1 = [-\pi/2,\pi/2]$,
        (b) $S_2 = [-\pi/3,\pi/3]$ and
        (c) $S_3 = [-\pi/4,\pi/4]$ at $k=4$. 
        Red crosses: true source; blue circles: DSM peaks. 
    }
    \label{fig:error_analysis_demo}
\end{figure}

We now turn to the more challenging task of localizing multiple point sources with the DSM. 
The previous section demonstrated that sparse, single-frequency measurements are already adequate for accurate single-source localization; however, the same strategy degrades noticeably when several sources are present. 
For this reason, most DSM variants (see, for example, \cite{li2021quality,liu2022deterministic,zhang2018locating}) rely on measurements acquired at multiple wavenumbers to stabilize the reconstruction. 
To quantify the contribution of frequency diversity, we examine how the cardinality of the wavenumber set $\mathcal{A}$  influences the localization performance in the following example.

\begin{example} \label{exa:10_measurements}
In this example, three point sources ($N=3$) with  identical magnitude $\lambda_j = 6$ ($j = 1,2,3$) are placed at
    \begin{equation*}
        z_1 = (-1.40,\, 1.05), \quad
        z_2 = (0.16,\, 1.56), \quad
        z_3 = (\,1.97,\,-0.37),
    \end{equation*}
    all measured in the same non-dimensional units as the wavelength.
    
    The scattered field is collected on a partial measurement curve 
    \begin{equation*}
        \Gamma = \bigl\{\, (R\cos\theta,\;R\sin\theta)\; \bigl| \; R = 6.5,\; \theta\in[-\frac{\pi}{2},\,\frac{\pi}{2}] \bigr\},
    \end{equation*}
    i.e. a semicircle with radius $6.5$ centered at the origin, which spans an aperture from $-\pi/2$ to $\pi/2$. 
    We sample $M = 10$ equal-angle measurements on $\Gamma$, creating a representative sparse configuration.
    
    We define six monotonically increases wavenumber sets
    \begin{equation*}
        \mathcal{A}_j=\left\{ x\,|\,x=4+\ell,\, \ell,j\in \mathbb{N} ,\, 0\leq\ell<j \right\} ,\qquad j=1,\cdots ,6,
    \end{equation*}
    so that frequencies are removed one by one while all other parameters remain fixed. 
    The DSM indicator $I_{{DSM}}(x)$ is evaluated on a uniform Cartesian grid over the domain $V = [-2, 2]\times[-2,2]$, which fully encloses the true sources.
\end{example}

As shown in Figure \ref{fig:sensor_10_change_wavenumber}, as the number of available wavenumbers is reduced, the peaks of DSM indicator exhibit increasing deviations from the true source locations. 
When only a single wavenumber ($k=4$) is available, although selecting the three highest peaks can mitigate the risk of completely missing all sources, spurious peaks frequently attain indicator values comparable to, or even higher than, those of the true sources, complicating reliable source identification.

\begin{figure}[htb]
    \centering
    \includegraphics[width=0.95\textwidth]{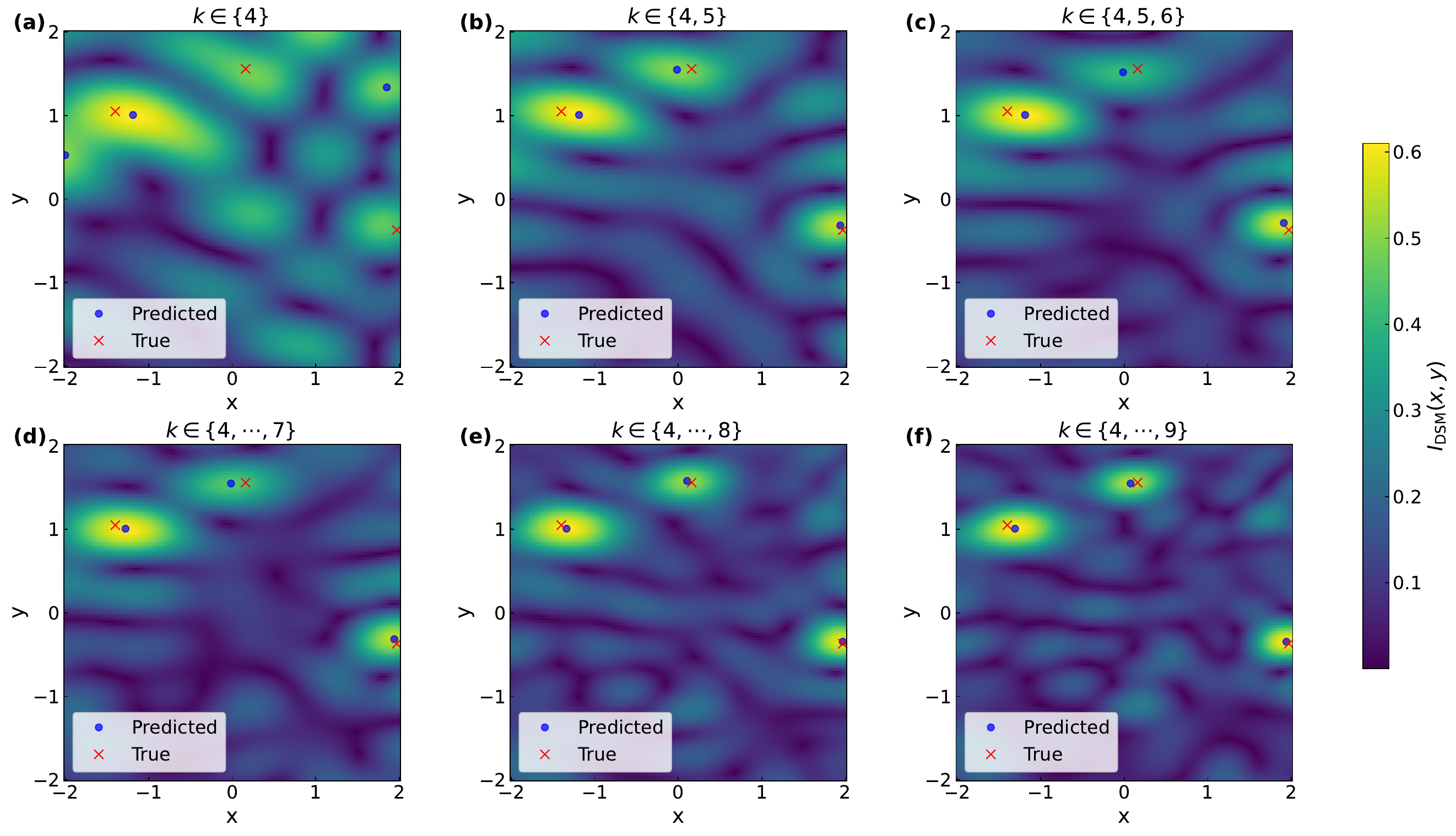} 
    \caption{
        \textbf{Three-source DSM with sparse angles (Example  \ref{exa:10_measurements}, $M=10$)}. Indicator maps for wavenumber sets (a) $\mathcal{A}_1$ to (f) $\mathcal{A}_6$. Red crosses: truth; blue circles: detected peaks. As the number of wavenumbers decreases, peaks drift away from the true locations.
    }
\label{fig:sensor_10_change_wavenumber}
\end{figure}

To examine the impact of sampling density on source localization, we consider another example in which only the number of measurement points is changed, while all other parameters are held fixed.

\begin{example} \label{exa:128_measurements}
    We repeat Example \ref{exa:10_measurements} with the number of measurement angles increased from $M = 10$ to $M =128$. 
    All other parameters remain unchanged.
\end{example}

Figure \ref{fig:sensor_128_change_wavenumber} displays DSM indicator maps reconstructed from $M=128$ measurement angles in Example \ref{exa:128_measurements}. 
In comparison with the 10-angle results reported in Example \ref{exa:10_measurements}, two systematic improvements are evident.
For every set $\mathcal A_j$ ($j=1,\dots,6$), the indicator maxima lie closer to the true source coordinates, and spurious extrema are suppressed. 
The single-wavenumber case $k=4$, which failed to yield correct location in Example \ref{exa:10_measurements}, now recovers all three sources.
The mean absolute localization error
\begin{equation}
    \varepsilon_j
    \;=\;
    \frac{1}{N} \sum_{i=1}^{N}
    \bigl\|\hat{\mathbf x}^{\,(j)}_i-\mathbf x_i\bigr\|_2,
    \qquad
    j=1,\dots,6,
    \label{eq:mae}
\end{equation}
is summarized in Table \ref{tab:mae_10_128}. 
Moving from a $10$-angle to a $128$-angle view reduces $\varepsilon_j$ by approximately one order of magnitude; the mean absolute error remains below $0.155$ even for $\mathcal A_1=\{4\}$.

\begin{figure}[htb]
    \centering
    \includegraphics[width=0.95\textwidth]{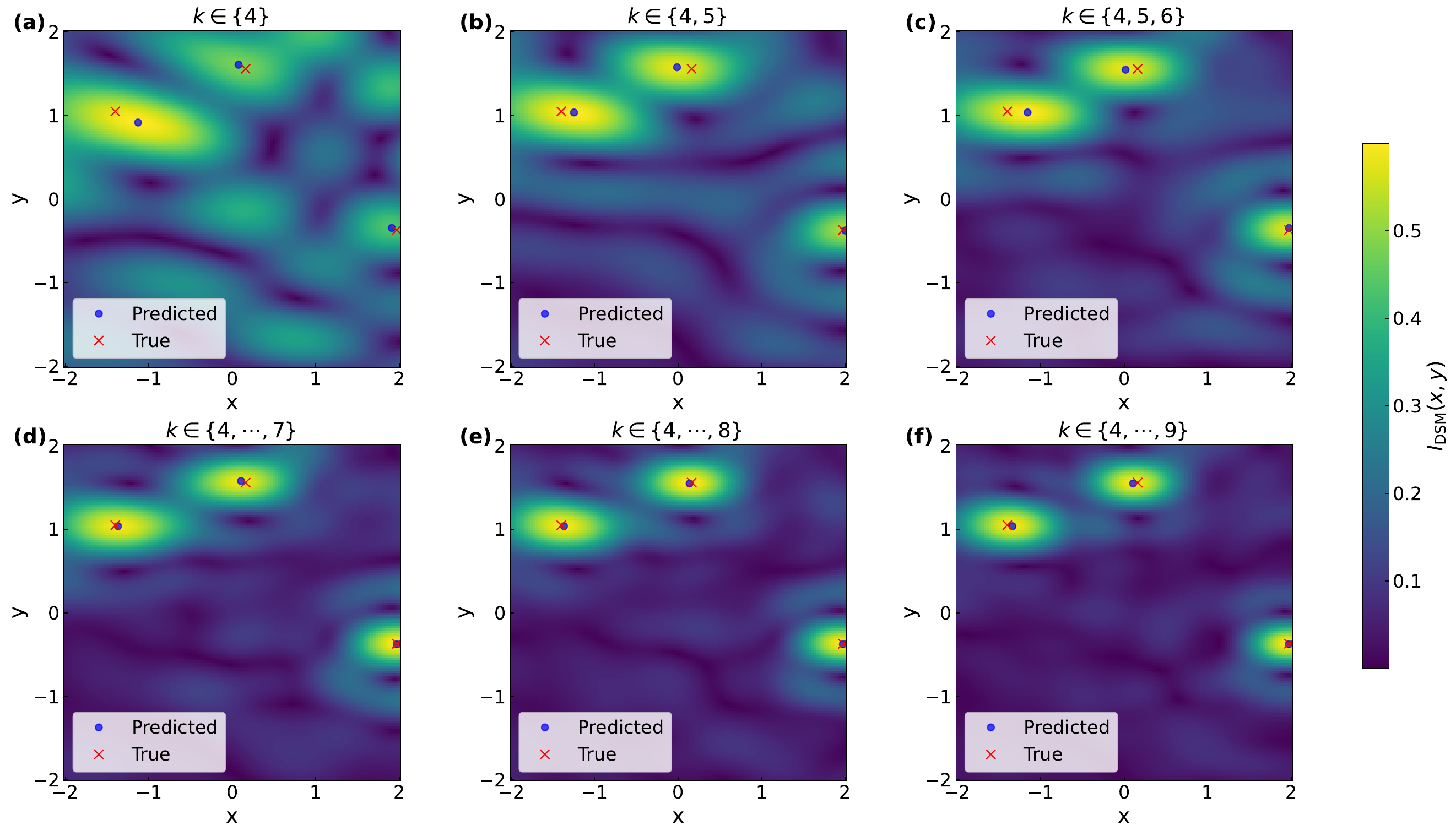}
    \caption{
        \textbf{Three-source DSM with dense angles (Example \ref{exa:128_measurements}, $M=128$).} Indicator maps for wavenumber sets (a) $\mathcal{A}_1$ to (f) $\mathcal{A}_6$. Red crosses: truth; blue circles: detected peaks. Dense angular sampling brings peaks closer to the ground truth and suppresses spurious extrema.
    }
    \label{fig:sensor_128_change_wavenumber}
\end{figure}


\begin{table}[htb]
    \caption{
        \textbf{Mean absolute error of DSM.} Mean absolute localization error $\varepsilon_j$ for the six wavenumber sets in Example \ref{exa:10_measurements} ($M=10$) and Example \ref{exa:128_measurements} ($M=128$).
    }
    \label{tab:mae_10_128}
    \centering
    \begin{tabular}{ccccccc}
        \toprule
        & $\mathcal A_1$ & $\mathcal A_2$ & $\mathcal A_3$
        & $\mathcal A_4$ & $\mathcal A_5$ & $\mathcal A_6$\\
        \midrule
        $M = 10$ & 1.322 & 0.158 & 0.150 & 0.125 & 0.065 & 0.061\\
        $M = 128$ & 0.155 & 0.124 & 0.122 & 0.033 & 0.024 & 0.028\\
        \bottomrule
    \end{tabular}
\end{table}

When measurements at multiple wavenumbers are available, the enhanced stability we observe is consistent with Theorem \ref{thm:uniqueness}, which guarantees uniqueness once the wavenumber set size $|\mathcal A_j|$ is sufficiently large. 
Remarkably, dense angular sampling improves the reconstruction accuracy across all wavenumber set, and the improvement is particularly pronounced in the single-wavenumber case, indicating that high measurement density can compensate for sparsity.

Taken together, these observation motivate a practical design principle for limited aperture inverse source problem: when the aperture cannot be widened, increasing the sampling density along the available curve can still ensure high-fidelity localization in both multi- and single-wavenumber settings.
To explore the mechanism behind this improvement, Figure \ref{fig:compare_sensor_10_128} plots the measured field values on the measurement curve $R=6.5$ for the single-frequency cases $k = 4,6,8$ under two sampling settings,
$M=10$ (as scatters) and $M=128$ (as curve). The 10-angle data are visibly under-sampled, leaving large portions of the oscillatory waveform unrecorded; even at the lowest wavenumber $k=4$ fine-scale features remain unresolved.
Dense sampling therefore recovers the missing physical information by resolving finer details in the measurement curve, providing the DSM with a more faithful representation of the scattered field and thereby enabling accurate source localization.

\begin{figure}[htb]
    \centering
    \includegraphics[width=0.95\textwidth]{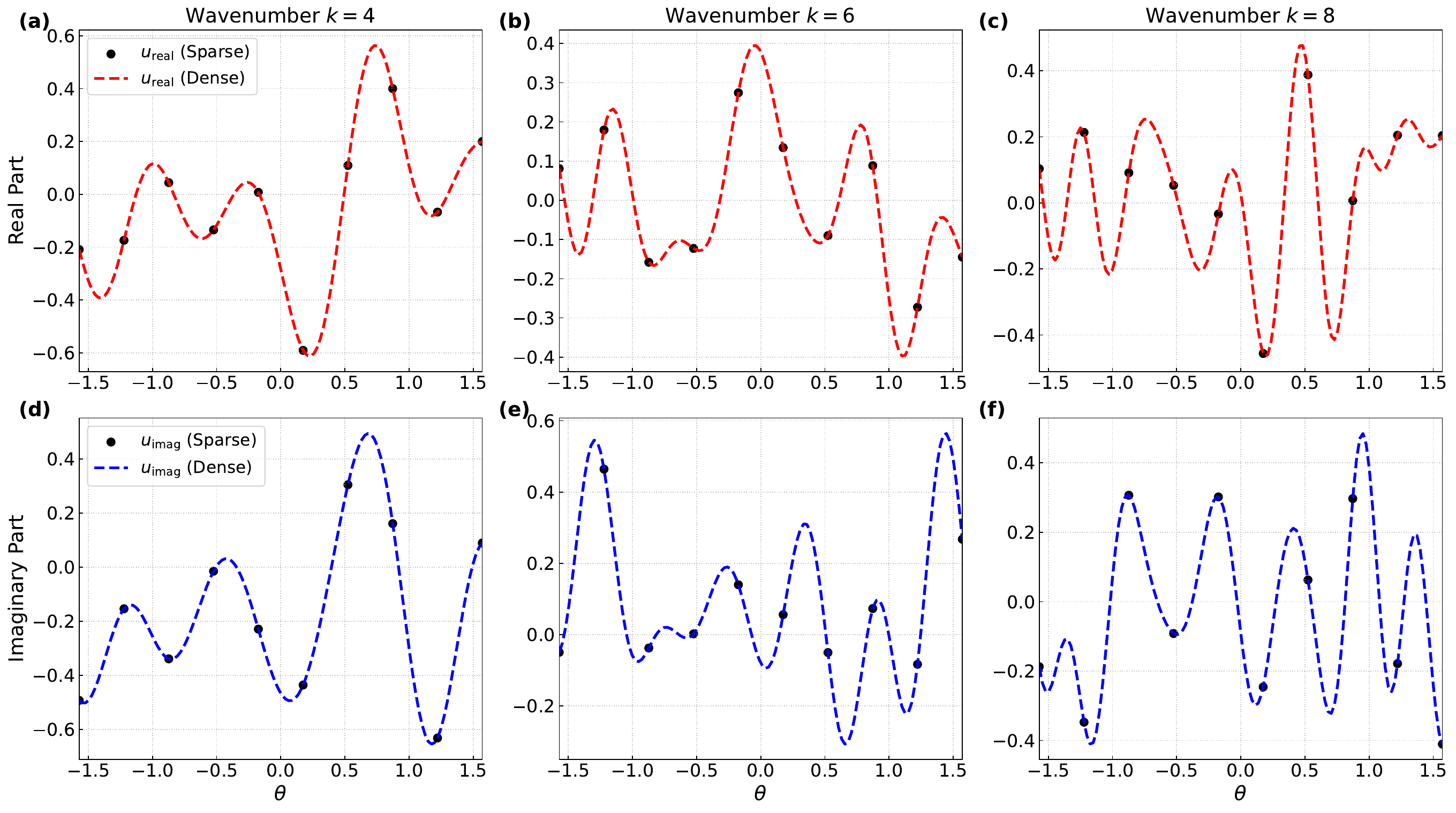}
    \caption{
        \textbf{Measured field along the measurement curve.} (a-c): real part; (d-f): imaginary part at $k=4,6,8$. Black dots: sparse ($M=10$) measurements; dashed curves: dense ($M=128$) measurements. Sparse measurements under-resolve oscillations, while dense measurements capture the waveform and supplie richer information to DSM.
    }
    \label{fig:compare_sensor_10_128}
\end{figure}

\subsection{Neural operator approach for inverse source problem}\label{sec2.3}

Inverse source localization based on limited measurements is notoriously ill-posed when the sensors cover only a partial aperture and the number of sensors is small. 
Motivated by the observations in Examples~\ref{exa:10_measurements}-\ref{exa:128_measurements} which indicate that classical solvers like DSM can enhance the localization quality when fed with sufficiently rich wavefield information, 
we aim to design a reconstruction scheme that converts sparse measurements on partial aperture into a reliable dense trace, which can then be fed to a classical inverse solver to remedy the instability and accuracy loss associated with sparse and partial-aperture data. 
With this guiding idea, we now formalize the reconstruction operator and our framework.

Following the notation in Section~2.1, the unknown $N$ point sources $\{z_j\}_{j=1}^N$ are supported in the sampling domain $V$, and sensors lie on the measurement curve $\Gamma$. 
A partial aperture $\Gamma$ carries a fixed array of $M$ sensors (sparse) $\Gamma_h:=\{x_{1,\mathrm{sen}},\ldots,x_{M,\mathrm{sen}}\} $. 
For any configuration of sources inside $V$, the scattered field $u:\Gamma\to\mathbb C$ is uniquely determined and the sensors record $u_{m,\mathrm{sen}} = u(x_{m,\mathrm{sen}})$, $m=1,\ldots,M$. 
For every wavenumber $k>0$, the field satisfies the Helmholtz system \eqref{eq:helmholtz_a}-\eqref{eq:helmholtz_b}. 

Let $\mathbf u_{\mathrm{sen}} := (u_{1,\mathrm{sen}},\ldots,u_{M,\mathrm{sen}})\in\mathbb C^M$. 
We seek a reconstruction operator $\mathcal T:\mathbb C^M\to L^2(\Gamma)$ such that
\begin{equation}\label{eq:reconstruction_problem}
\mathcal T \mathbf u_{\mathrm{sen}}
:= \arg\min_{v\in\mathcal V}\;\|v-u\|_{L^2(\Gamma)}^2,
\end{equation}
where $\mathcal V := \left\{\, v\in L^2(\Gamma) \mid v(x_{m,\mathrm{sen}})=u_{m,\mathrm{sen}},\; m=1,\ldots,M \right\}.$
For brevity we still write $\mathcal T \mathbf u_{\mathrm{sen}}$ for the function $u(x)$ on $\Gamma$. 
Problem \eqref{eq:reconstruction_problem} formalizes sparse-to-dense completion as identifying a reconstruction operator defined on $\Gamma$.

In practice the exact operator $\mathcal T$ is unknown. 
We therefore approximate the mapping by a neural network with parameters $\eta$,
\begin{equation}\label{eq:approximate_by_neural_network}
(\mathbf u_{\mathrm{sen}}, x) \longmapsto [\mathcal T_{\eta} \mathbf u_{\mathrm{sen}}](x),\qquad x\in\Gamma.
\end{equation}
We propose a modular hybrid framework that couples a data-driven reconstruction operator with a classical inverse solver; see Figure~\ref{fig:framework}. 
The operator $\mathcal T_\eta$ is trained offline to approximate \eqref{eq:approximate_by_neural_network}. 
At inference step, a new sparse vector $\mathbf u_{\mathrm{sen}}^{(\mathrm{new})}$ on $\Gamma_h$ is first interpolated (DeepONet \cite{lu2021learning} in this work) on the entire aperture $\Gamma$; the completed trace is then consumed by a conventional inverse solver (DSM in this work). 
This decoupling compensates for information loss caused by sparse measurement while preserving the interpretability and theoretical guarantees of the classical solver. 
Since the two modules are decoupled, either the operator learner (e.g., DeepONet, FNO, Transformer- or graph-based architectures) or the inverse solver (e.g., DSM, optimization approaches, Bayesian methods) can be substituted without modifying the workflow.
We summarize the workflow as following, and for deatils of offline training can be seen at Algorithm \ref{alg:offline}:

\textbf{Step 0 (Offline training).}
Generate synthetic dataset from diverse point-source configurations and train a neural operator $T_\eta$ to approximate the mapping \eqref{eq:approximate_by_neural_network}. 
The optimized surrogate $T_{\eta^\ast}$ enables real-time sparse-to-dense interpolation on $\Gamma$.

\textbf{Step 1 (Interpolation on $\Gamma$).}
Given an unseen sparse measurement $\mathbf u_{\mathrm{sen}}^{(\mathrm{new})}$ acquired on $\Gamma_h$, evaluate
$\tilde{\mathbf u} := T_{\eta^\ast} \mathbf u_{\mathrm{sen}}^{(\mathrm{new})}$
to obtain a dense, self-consistent trace over the entire aperture $\Gamma$.

\textbf{Step 2 (Inverse solver).}
Feed the dense trace $\tilde{\mathbf u}$ to a conventional inverse solver to localize the unknown sources.

\begin{remark}
    \textbf{DeepONet parameterization:}\,\,\,
    In this work, we represent $\mathcal T_\eta$ by an unstacked deep neural network (DeepONet) consisting of a ``branch net'' and a ``trunk net'' \cite{lu2021learning}. 
    The branch net takes $\mathbf u_{\mathrm{sen}}$ and outputs $(b_1,\ldots,b_q)\in\mathbb R^q$; the trunk net takes a query $x\in\Gamma$ and outputs $(t_1(x),\ldots,t_q(x))\in\mathbb R^q$. 
    Merging the two embeddings by an inner product yields a continuous prediction on $\Gamma$:
    \begin{equation}
    [\mathcal T_\eta \mathbf u_{\mathrm{sen}}](x) = \sum_{k=1}^q b_k(\mathbf u_{\mathrm{sen}})\, t_k(x),
    \end{equation}
    thereby accomplishing fast sparse-to-dense interpolation of the scattered field over the aperture.
\end{remark}

\begin{remark}
    \textbf{Dataset generation:} We draw $N_{\mathrm{cfg}}$ independent point-source configuration $\{z_j^{(\ell)}\}_{j=1}^N \subset V$, $\ell=1,\ldots,N_{\mathrm{cfg}}$, and select a wavenumber $k>0$. For each configuration we evaluate the field at the physical sensors to form
    \begin{equation*}
        \mathbf u_{\mathrm{sen}}^{(\ell)} := \bigl(u^{(\ell)}(x_{1,\mathrm{sen}}),\ldots,u^{(\ell)}(x_{M,\mathrm{sen}})\bigr)\in\mathbb C^M .
    \end{equation*}
    We then sample $N_{\mathrm{aux}}$ auxiliary locations $\{x_{p,\mathrm{aux}}^{(\ell)}\}_{p=1}^{N_{\mathrm{aux}}}\subset\Gamma$ and record $u_{p,\mathrm{aux}}^{(\ell)} := u^{(\ell)}(x_{p,\mathrm{aux}}^{(\ell)})$. This yields the offline dataset of triples
    \begin{equation*}
        \mathcal D \;:=\; \Bigl\{\bigl(\mathbf u_{\mathrm{sen}}^{(\ell)},\,x_{p,\mathrm{aux}}^{(\ell)},\,u_{p,\mathrm{aux}}^{(\ell)}\bigr)\Bigr\}_{\ell=1,\ldots,N_{\mathrm{cfg}}, \, p=1,\ldots,N_{\mathrm{aux}}}
        \;\subset\; \mathbb C^M \times \Gamma \times \mathbb C ,
    \end{equation*}
    with $N_{\mathrm{cfg}}N_{\mathrm{aux}}$ samples. 
\end{remark}

\begin{remark}
    \textbf{Loss function: }
    During training we draw mini-batches $\mathcal B\subset\mathcal D$ of size $N_{\mathrm{mb}}:=|\mathcal B|$ and minimize the mean-squared error
    \begin{equation*}
        \mathcal L(\eta)
        = \frac{1}{N_{\mathrm{mb}}}\sum_{(\ell,p)\in\mathcal B}
        \bigl|\,[\mathcal T_\eta \mathbf u^{(\ell)}_{\mathrm{sen}}]\bigl(x^{(\ell)}_{p,\mathrm{aux}}\bigr) - u^{(\ell)}_{p,\mathrm{aux}} \,\bigr|^2 .
    \end{equation*}
\end{remark}

The optimized surrogate $\mathcal T_{\eta^\ast}$ reconstructs high-fidelity virtual measurements and substantially enhances the accuracy and robustness of inverse localization from sparse partial-aperture data.

\begin{algorithm}[H]
\caption{Offline training of the reconstruction operator $\mathcal T_{\eta^\ast}$}
\label{alg:offline}
\begin{algorithmic}
\Require Domain $V$, curve $\Gamma$, wavenumber $k>0$; numbers $N_{\mathrm{cfg}}$, $N_{\mathrm{aux}}$; sensor set $\{x_{m,\mathrm{sen}}\}_{m=1}^{N_{\mathrm{sen}}}\subset\Gamma_h$.
\State Initialize network parameters $\eta$ of DeepONet $\mathcal T_\eta$ (branch/trunk).
\For{$\ell = 1,\ldots,N_{\mathrm{cfg}}$}
    \State Sample a configuration of $N$ point sources $\{(z^{(\ell)}_j,\lambda^{(\ell)}_j)\}_{j=1}^N \subset V$.
    \State Solve the forward model at sensors to get $\mathbf u^{(\ell)}_{\mathrm{sen}}$.
    \State Sample auxiliary locations $\{x^{(\ell)}_{p,\mathrm{aux}}\}_{p=1}^{N_{\mathrm{aux}}}\subset \Gamma$ and record $\{ u^{(\ell)}_{p,\mathrm{aux}}\}_{p=1}^{N_{\mathrm{aux}}}$.
\EndFor
\State Build the offline dataset $\mathcal D \;:=\; \Bigl\{\bigl(\mathbf u_{\mathrm{sen}}^{(\ell)},\,x_{p,\mathrm{aux}}^{(\ell)},\,u_{p,\mathrm{aux}}^{(\ell)}\bigr)\Bigr\}_{\ell=1,\ldots,N_{\mathrm{cfg}}, \, p=1,\ldots,N_{\mathrm{aux}}}$.

\While{not converged}
    \State Sample a mini-batch $\mathcal B\subset\mathcal D$ of size $N_{\mathrm{mb}}$.
    \State Compute the MSE loss $\mathcal L(\eta)
        = \frac{1}{N_{\mathrm{mb}}}\sum_{(\ell,p)\in\mathcal B}
        \bigl|\,[\mathcal T_\eta \mathbf u^{(\ell)}_{\mathrm{sen}}]\bigl(x^{(\ell)}_{p,\mathrm{aux}}\bigr) - u^{(\ell)}_{p,\mathrm{aux}} \,\bigr|^2$.
    \State Update $\eta$ by AdamW with Cosine annealing warm restarts schedule.
\EndWhile
\State \textbf{return} Trained surrogate $\mathcal T_{\eta^\ast}$.
\end{algorithmic}
\end{algorithm}

\begin{figure}[t]
\includegraphics[width=1\linewidth]{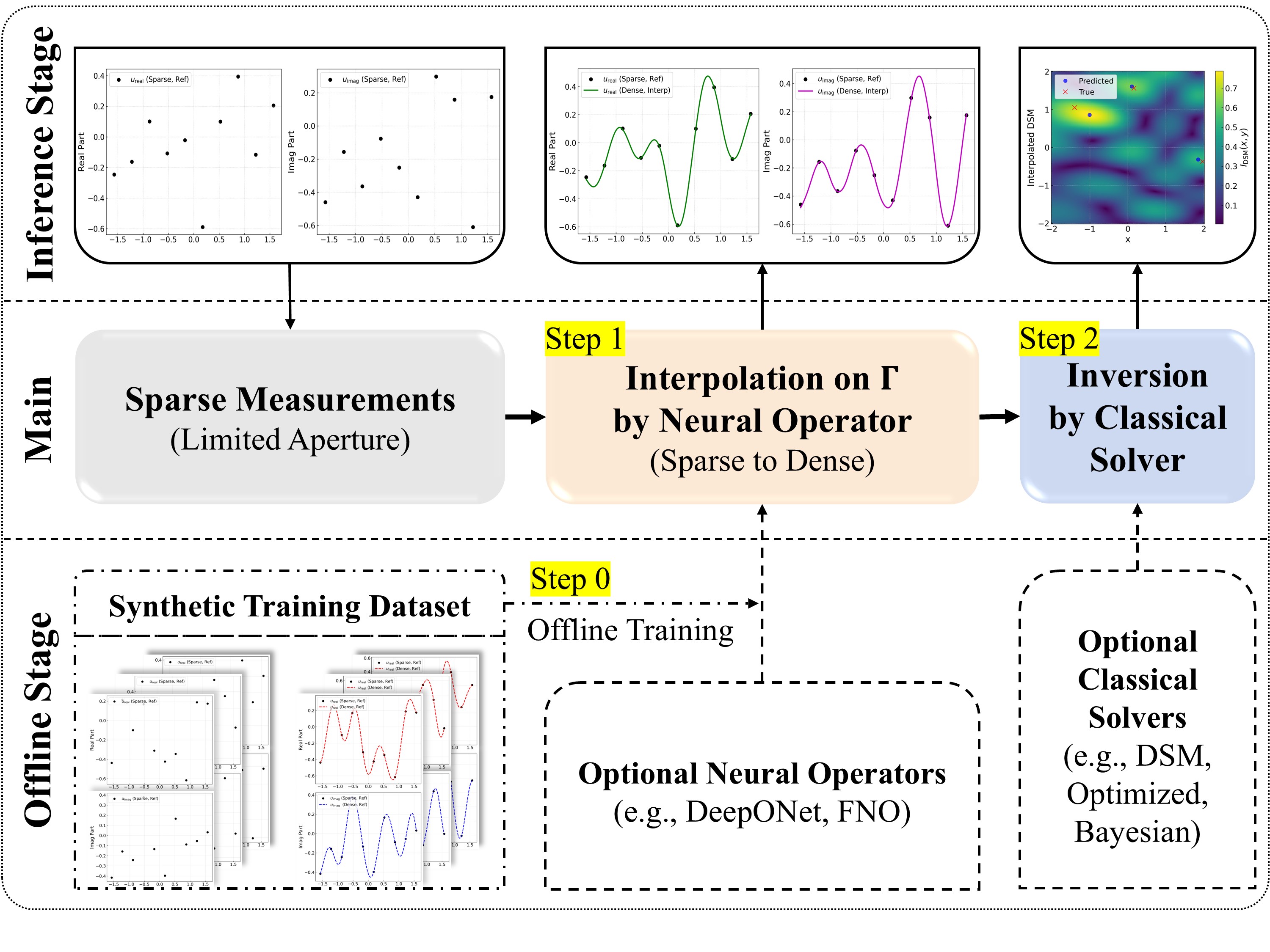}
    \caption{
        \textbf{Operator-learning-driven framework for sparse partial apertures.} Sparse boundary measurements are interpolated to a dense, self-consistent trace by a pre-trained neural operator; the dense data are then consumed by a classical inverse solver. The reconstruction and inversion modules are decoupled and can be swapped with alternative neural operators and classical solvers.
    }
    \label{fig:framework}
\end{figure} 


\section{Numerical Experiments}\label{sec3}

Let $V=[-2,2]^{2}\subset\mathbb R^{2}$. 
We consider $N\in\{2,3\}$ acoustical point sources $\{(z_{j}, \lambda_{j})\}_{j=1}^{N}$, where the centers $z_{j}$ are drawn independently and uniformly from $V$ and the magnitudes  $\lambda_{j}\sim\mathcal U(5,7)$.
The scattered field $u$ satisfies the homogeneous Helmholtz equation \eqref{eq:helmholtz_a} with wavenumber $k=4$ and the Sommerfeld radiation condition \eqref{eq:helmholtz_b}.

Measurements are acquired on the radius-$R$ circle $\Gamma=\{x\in\mathbb R^{2}\;:\;|x|=R\}$ with $R=6.5$, but only within one of the three partial apertures
\begin{equation*}
    S_{1}=\left[-\frac{\pi}{2},\frac{\pi}{2}\right],\quad
    S_{2}=\left[-\frac{\pi}{3},\frac{\pi}{3}\right],\quad
    S_{3}=\left[-\frac{\pi}{4},\frac{\pi}{4}\right].
\end{equation*}

We deploy $N_{\mathrm{sen}}^{(q)}, \,q=1,2,3$ equi-angular sensors for the sector $S_{q}$, with $\bigl(N_{\mathrm{sen}}^{(1)},N_{\mathrm{sen}}^{(2)},N_{\mathrm{sen}}^{(3)}\bigr)=(10,8,6)$.
For $\theta_{m}\in S_{q}$ the $m$-th measurement is $u_{m}=u(R\cos\theta_{m},R\sin\theta_{m}) + \epsilon_{m}$, where $\epsilon_{m}\sim\mathcal N_{\mathbb C}\bigl(0,0.05^{2}|u|^{2}\bigr)$ is synthetic noise.

For each aperture, we build $N_{\mathrm{cfg}}=10,000$ source configurations.
For the sparse measurement vector $\mathbf u_{\mathrm{sen}}\in\mathbb C^{N_{\mathrm{sen}}^{(q)}}$ , we randomly record $N_{\mathrm{aux}}=128$ auxiliary samples $u(R\cos\varphi_{p},R\sin\varphi_{p})$ at angles $\varphi_{p}, p = 1, \cdots, N_{\rm aux}$.
The resulting data set of $1.28\times10^{6}$ triplets is used for training; An independent test set of equal size is reserved for evaluation.

For each aperture, we train an independent Deep Operator Network $\mathcal T_{\eta}^{(q)}$.
Since the number of sensors varies with $q$, the branch net first applies a  two-layer MLP $(N_{\rm sen}^{(q)}\!-\!256\!-\!256)$ to every sensor measurement.
The trunk net uses a three-layer MLP $(1\!-\!256\!-\!256-\!256)$ to embed a query angle $\varphi$.
Both networks employ $\tanh$ activations.
The inner product of the two 256-dimensional embeddings yields the prediction $ \widehat u(\varphi) = [\mathcal T_{\eta}^{(q)}\mathbf u_{\mathrm{sen}}](\varphi)$.

Parameters are optimized with AdamW \cite{loshchilov2017decoupled} (initial learning rate $10^{-3}$, weight decay $10^{-4}$), and a cosine-annealing schedule with warm restarts \cite{loshchilov2016sgdr} ($T_{0}=1000$, $T_{\mathrm{mult}}=2$, $\eta_{\min}=10^{-6}$).
Each mini-batch aggregates $50,000$ $(\mathbf u_{\mathrm{sen}},\varphi,u_{\mathrm{aux}})$ triplets drawn across configurations and query points.
Training converges after $10,000$ iterations, requiring approximately two mins on one NVIDIA RTX3090.
A complete list of parameters is reported in Table \ref{tab:params}.

\begin{table}[tb]
\centering
\caption{\textbf{Numerical constants used throughout the experiments.}}
\label{tab:params}
\begin{tabular}{@{}llc@{}}
\toprule
\textbf{Category} & \textbf{Parameter} & \textbf{Value / Range} \\ \midrule

\multirow{4}{*}{\rm Physical model} 
 & Wavenumber $k$ & $4$ \\ 
 & Domain $V$ & $[-2,2]^{2}$ \\ 
 & Source amplitude $\lambda_{j}$ & ${\mathcal U(5,7)}$ \\ 
 & Number of sources $N$ & $2$ \text{or} $3$ \\ \midrule
 
\multirow{6}{*}{\rm Dataset $^{\dagger}$} 
 & Radius $R$ & $6.5$ \\
 & Apertures $S_{1},S_{2},S_{3}$ & $\bigl[-\frac{\pi}{2},\frac{\pi}{2}\bigr]$, $\bigl[-\frac{\pi}{3},\frac{\pi}{3}\bigr]$, $\bigl[-\frac{\pi}{4},\frac{\pi}{4}\bigr]$ \\
 & Sensors $N_{\rm sen}^{(1)},N_{\rm sen}^{(2)},N_{\rm sen}^{(3)}$ & $10,\,8,\,6$ \\
 & Configurations $N_{\rm cfg}$ & $10{,}000$ \\
 & Auxiliary samples $N_{\rm aux}$ & $128$ \\ \midrule
 
\multirow{3}{*}{\rm DeepONet}
 & Branch net & $(N_{\rm sen},\,256,\,256)$ \\
 & Trunk net & $(1,\,256,\,256,\,256)$ \\
 & Activation & ${\rm Tanh}$ \\ \midrule
 
\multirow{4}{*}{\rm Training}
 & Optimizer & AdamW \\
 & Initial learning rate & $10^{-3}$ \\
 & Cosine annealing & $T_{0}=1000$, $T_{\rm mult}=2$, $\eta_{\min}=10^{-6}$ \\
 & Warm restarts schedule &\\
 & Mini--batch size & $50{,}000$ triplets \\ \midrule
 
\multirow{2}{*}{\rm Runtime}
 & Iterations to converge & $10{,}000$ \\
 & Hardware & NVIDIA RTX3090 \\
\bottomrule
\multicolumn{3}{@{}p{0.85\linewidth}@{}}{\footnotesize%
$^{\dagger}$\;Measurements are restricted to one aperture at a time; three DeepONet models are trained independently.}
\end{tabular}
\end{table}

\subsection{Two-source localization}\label{sec3.1}

We begin our numerical investigation by fixing the number of point sources at $N=2$, and keeping all other parameters identical to those described in the experimental setup.

Figure~\ref{fig:loss_curves_2} displays the evolution of the training loss $\mathcal{L}$ over $10{,}000$ optimization steps for the three aperture configurations. 
In all cases, the loss decreases rapidly by approximately three orders of magnitude within the first $2000$ iterations and subsequently stabilizes around a low plateau. 
Notably, all three curves share an almost identical shape: after a drop of roughly three orders of magnitude during the first $2000$ iterations, the loss flattens at a similar low value, indicating the robustness of the training process with respect to aperture size. 

\begin{figure}[htb]
    \centering
    \includegraphics[width=0.95\textwidth]{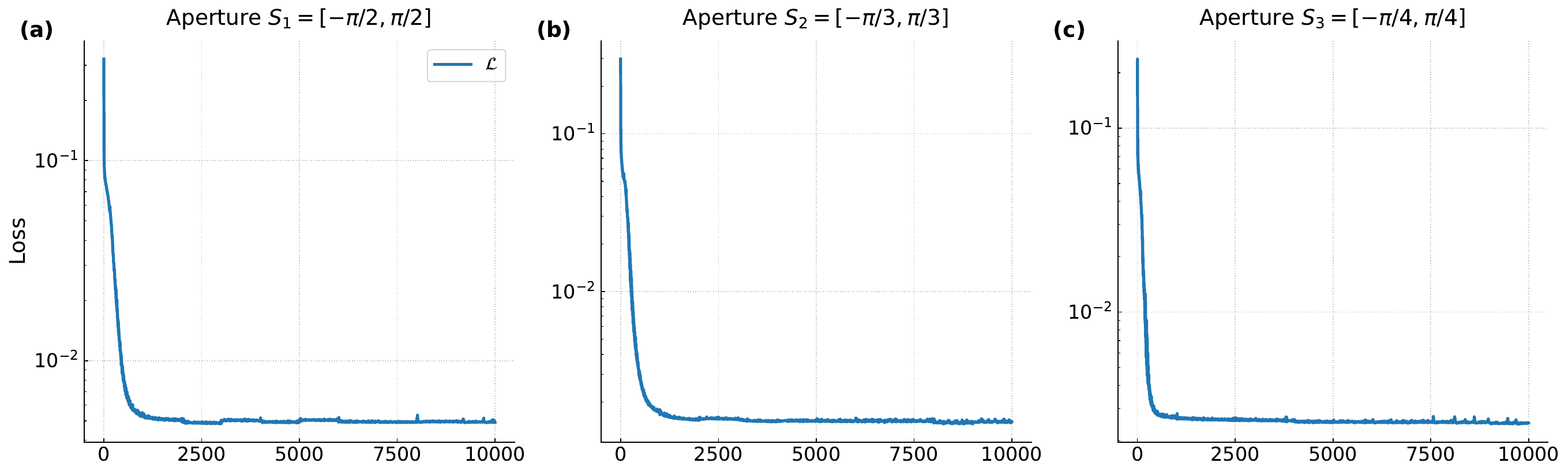}
    \caption{
        \textbf{Training curves for two-source interpolation ($N=2$).} Log-linear loss versus iteration for the three apertures $[-\pi/n,\pi/n],n=2,3,4$.
    }
        \label{fig:loss_curves_2}
\end{figure}

To evaluate the accuracy of the learned operator, we consider three apertures with random  two-source locations
\begin{align*}
    S_1:&\ \{(1.37,-0.35),\,(-0.83,-1.24)\},\\
    S_2:&\ \{(-1.09,-1.91),\,(-1.92,0.08)\},\\
    S_3:&\ \{(-0.26,-1.78),\,(-1.18,1.65)\}.
\end{align*}
Source amplitudes are drawn from $\mathcal U(5,7)$. 
Using a single wavenumber $k=4$ and the same sensor placement as in the training set, we collect $10$, $8$, and $6$ measurements on $S_1$, $S_2$, and $S_3$, respectively, and perform localization using only these measurements.
We then compare the predicted complex field $\widehat{u}(\theta)$ on each aperture $S_q$ against the densely sampled ground-truth field $u(\theta)$.
Figure~\ref{fig:compare_interp_measurement_2} confirms that the neural-operator interpolation result closely follows the reference in both real and imaginary parts. 

\begin{figure}[htb]
    \centering
    \includegraphics[width=0.95\textwidth]{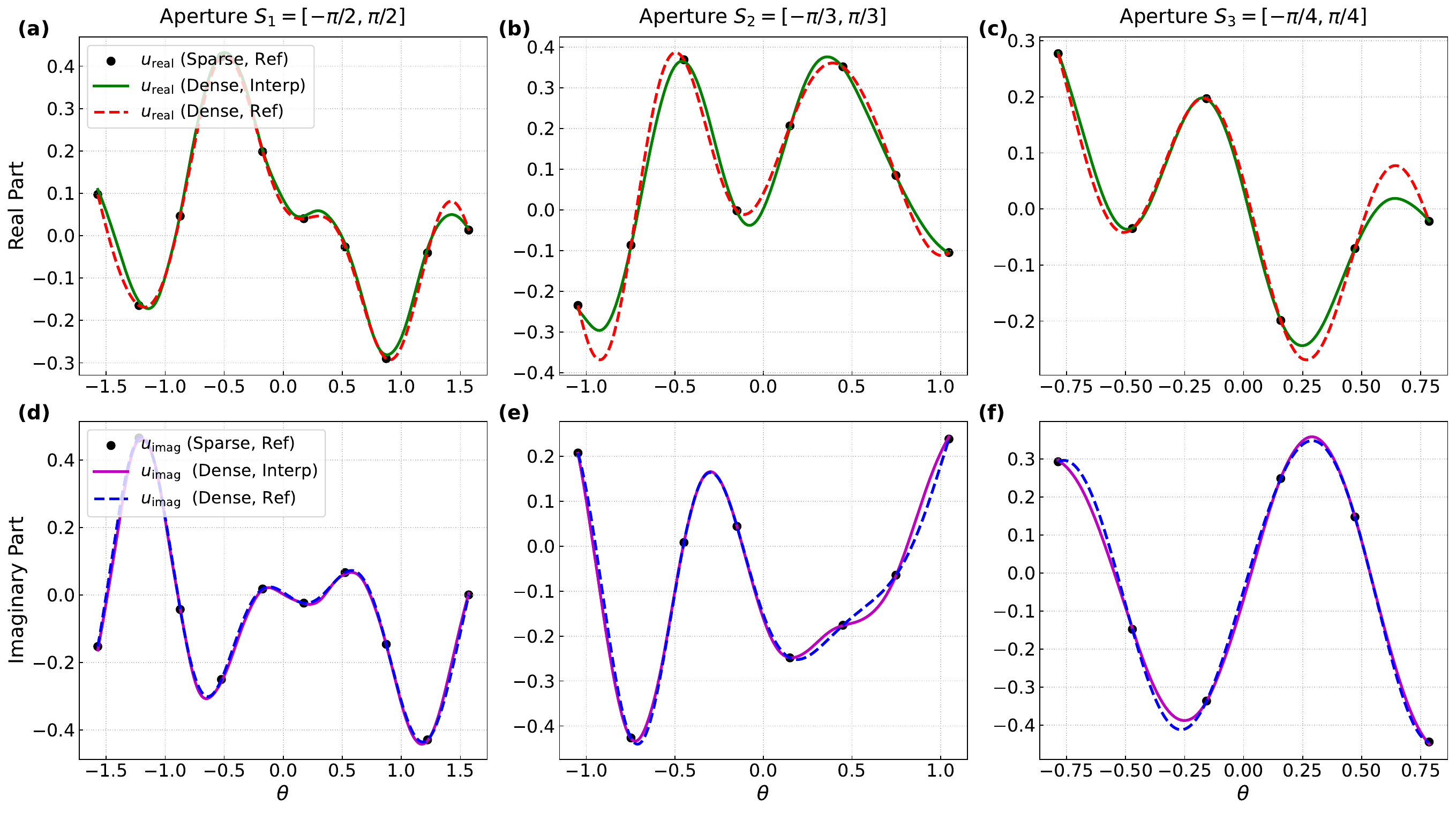}
    \caption{
        \textbf{Comparison of operator-interpolated dense traces with ground-truth data ($N=2$).} Real part (a-c) and imaginary part (d-f) of the predicted dense trace $\widehat u(\varphi)$ (solid), the reference dense trace $u(\varphi)$ (dashed), and the $M$ measurements (dots) for the
        three apertures.
    }
    \label{fig:compare_interp_measurement_2}
\end{figure}

Figure~\ref{fig:compare_dsm_interpolated_dsm_2} illustrates how DeepONet-based interpolation affects DSM-based localization.
We compare DSM localization maps obtained from the original sparse data (top row) with those computed from the DeepONet-interpolated data (bottom row). 
For each aperture, the blue circles indicate the predicted source locations, and the red crosses denote the ground truth. 
Without interpolation, DSM mislocalizes at least one source, most severely for $S_2$ and $S_3$. 
With DeepONet interpolation, all sources are correctly identified and mean absolute localization error drops from $1.316$, $0.852$ and $1.333$ to $0.244$, $0.231$ and $0.114$, respectively. A detailed per-aperture comparison is summarized in Table~\ref{tab:error_comparison_2}.

\begin{table}[htb]
  \caption{\textbf{Error comparison of DSM and DeepONet-interpolated DSM ($N=2$).}
  Mean absolute error (MAE) with DSM from raw sparse data and DeepONet-interpolated data (ours).}
  \label{tab:error_comparison_2}
  \centering
  \setlength{\tabcolsep}{5pt}
  
  \resizebox{\textwidth}{!}{
  \begin{tabular}{l*{3}{ccc}}
    \toprule
    & \multicolumn{3}{c}{$S_1$} & \multicolumn{3}{c}{$S_2$} & \multicolumn{3}{c}{$S_3$} \\
    \cmidrule(lr){2-4}\cmidrule(lr){5-7}\cmidrule(lr){8-10}
    Method & $z_{\text{true}}$ & $z_{\text{pred}}$ & MAE
           & $z_{\text{true}}$ & $z_{\text{pred}}$ & MAE
           & $z_{\text{true}}$ & $z_{\text{pred}}$ & MAE \\
    \midrule
    DSM
      & \begin{tabular}{@{}c@{}}$(1.37,-0.35)$\\$(-0.83,-1.24)$\end{tabular}
      & \begin{tabular}{@{}c@{}}$(1.308,-0.375)$\\$(1.879,-1.218)$\end{tabular}
      & 1.316
      & \begin{tabular}{@{}c@{}}$(-1.09, -1.91)$\\$(-1.92,0.08)$\end{tabular}
      & \begin{tabular}{@{}c@{}}$(-0.466, -1.819)$\\$(-0.406 , -1.789)$\end{tabular}
      & 0.852
      & \begin{tabular}{@{}c@{}}$(-0.26,-1.78)$\\$(-1.18, 1.65)$\end{tabular}
      & \begin{tabular}{@{}c@{}}$(-0.496, -1.879)$\\$(-0.586, -1.909)$\end{tabular}
      & 1.333 \\
    \midrule
    Ours
      & \begin{tabular}{@{}c@{}}$(1.37,-0.35)$\\$(-0.83,-1.24)$\end{tabular}
      & \begin{tabular}{@{}c@{}}$(1.127,-0.406)$\\$(-0.406,-1.157)$\end{tabular}
      & 0.244
      & \begin{tabular}{@{}c@{}}$(-1.09,-1.91)$\\$(-1.92,0.08)$\end{tabular}
      & \begin{tabular}{@{}c@{}}$(-0.436, -1.789)$\\$(-2, 0.045)$\end{tabular}
      & 0.231
      & \begin{tabular}{@{}c@{}}$(-0.26,-1.78)$\\$(-1.18,1.65)$\end{tabular}
      & \begin{tabular}{@{}c@{}}$(-0.466, -1.939)$\\$(-1.067, 1.789)$\end{tabular}
      & 0.114 \\
    \bottomrule
  \end{tabular}
  }
\end{table}

These results demonstrate that operator-learning-driven interpolation significantly improves DSM localization accuracy.

\begin{figure}[htb]
    \centering
    \includegraphics[width=0.95\textwidth]{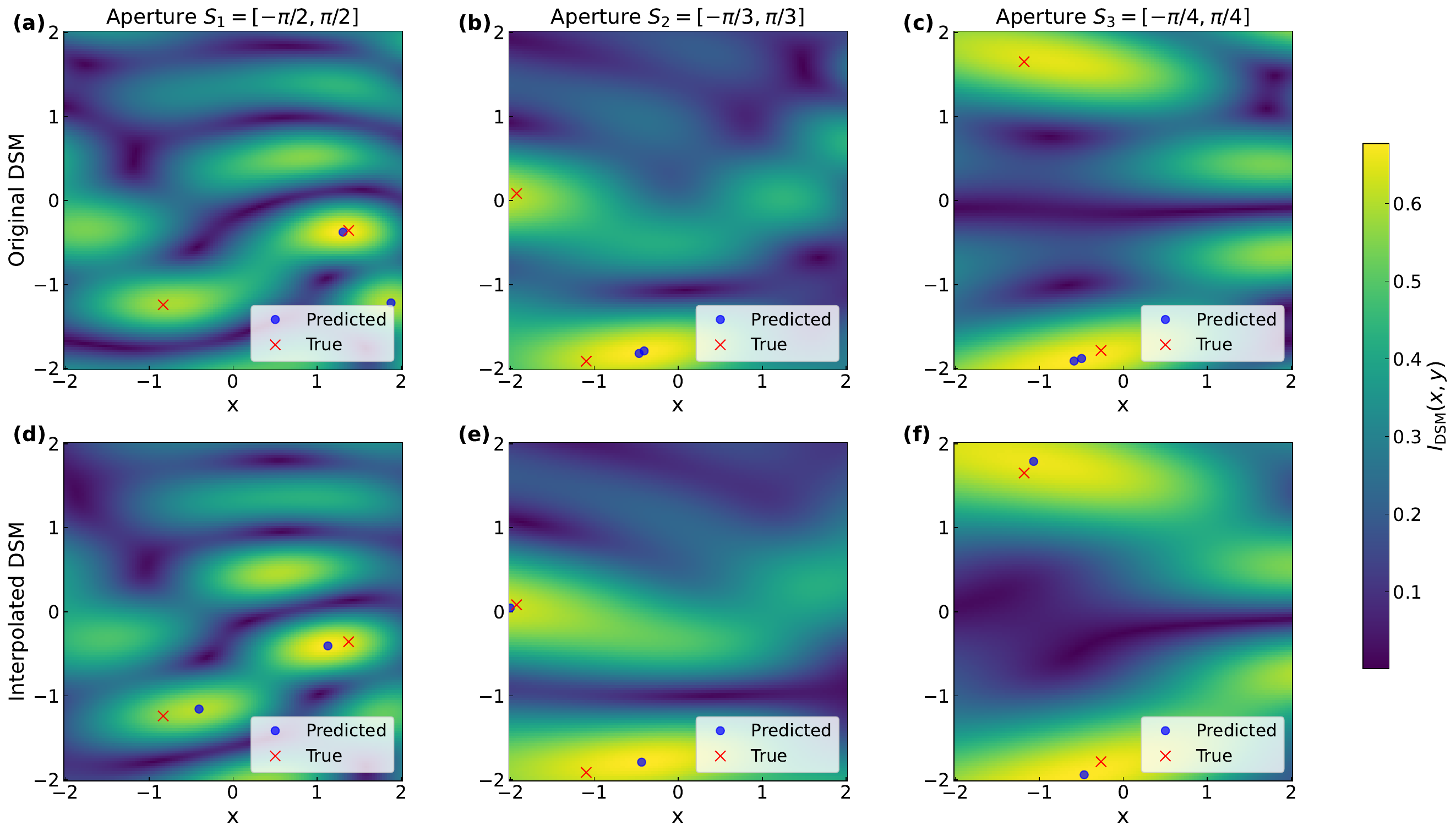}
    \caption{
        \textbf{Effect of operator interpolation on DSM source localization ($N=2$).} Top: DSM from raw sparse data; bottom: DSM from DeepONet-interpolated data; columns: apertures $[-\pi/n,\pi/n],n=2,3,4$. Red crosses: truth; blue circles: estimates. Mean absolute error drops from $1.316$, $0.852$, $1.333$ to $0.244$, $0.231$, $0.114$, respectively.
        }
    \label{fig:compare_dsm_interpolated_dsm_2}
\end{figure}

\subsection{Three-source localization}\label{sec3.2}

We now turn to the case of three point sources $N=3$, while keeping all other parameters unchanged to those described in the experimental setup.

Figure~\ref{fig:loss_curves_3} shows the training loss $\mathcal{L}$ over $10{,}000$ optimization steps for each of the three aperture configurations. 
A logarithmic ordinate highlights the convergence behavior. 
All three curves share the same shape: a drop of roughly three orders of magnitude in the first $2000$ iterations is followed by a common low plateau, which indicates that convergence speed and final loss are essentially insensitive to aperture width.

\begin{figure}[htb]
    \centering
    \includegraphics[width=\textwidth]{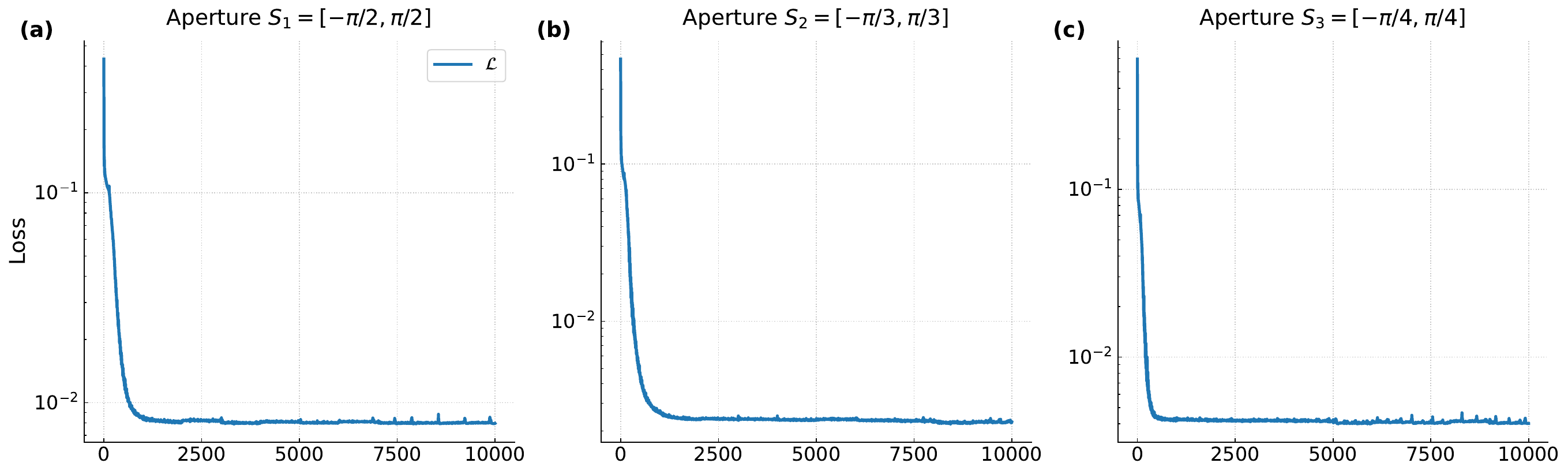}
    \caption{
        \textbf{Training curves for two-source interpolation ($N=3$).} Log-linear loss versus iteration for the three apertures $[-\pi/n,\pi/n],n=2,3,4$.
    }
    \label{fig:loss_curves_3}
\end{figure}

We test the random three-source configuration per aperture:
\begin{align*}
    S_1:&\ \{(1.61, 1.59), (0.64, -1.31), (-1.26, 0.63)\},\\
    S_2:&\ \{(1.20, -0.59), (-0.69, 1.96), (-1.43, -1.73)\},\\
    S_3:&\ \{(1.05, 1.38), (-0.66, -1.82), (-1.65, 0.38)\}.
\end{align*}
With $\lambda_j \sim \mathcal U(5,7)$, a single wavenumber $k=4$ and the same sensor placement as in training, we acquire $10$, $8$, and $6$ measurements on $S_1$, $S_2$, and $S_3$, respectively, and carry out localization using only these sparse measurements.
        
Figure~\ref{fig:compare_interp_measurement_3} confirms that the operator-generated dense trace follows the reference almost perfectly in both the real and imaginary parts. 
Figure~\ref{fig:compare_dsm_interpolated_dsm_3} illustrates how DeepONet-based interpolation affects DSM-based source localization. 
DSM localization maps obtained from the original sparse data (top row) are contrasted with those computed from the operator-completed data (bottom row). With DeepONet completion, all three sources are recovered and the mean absolute localization error falls from $0.573$, $0.728$ and $0.547$ to $0.094$, $0.243$ and $0.166$, respectively. A detailed per-aperture comparison is summarized in Table~\ref{tab:error_comparison_3}.

\begin{figure}[htb]
    \centering
    \includegraphics[width=0.95\textwidth]{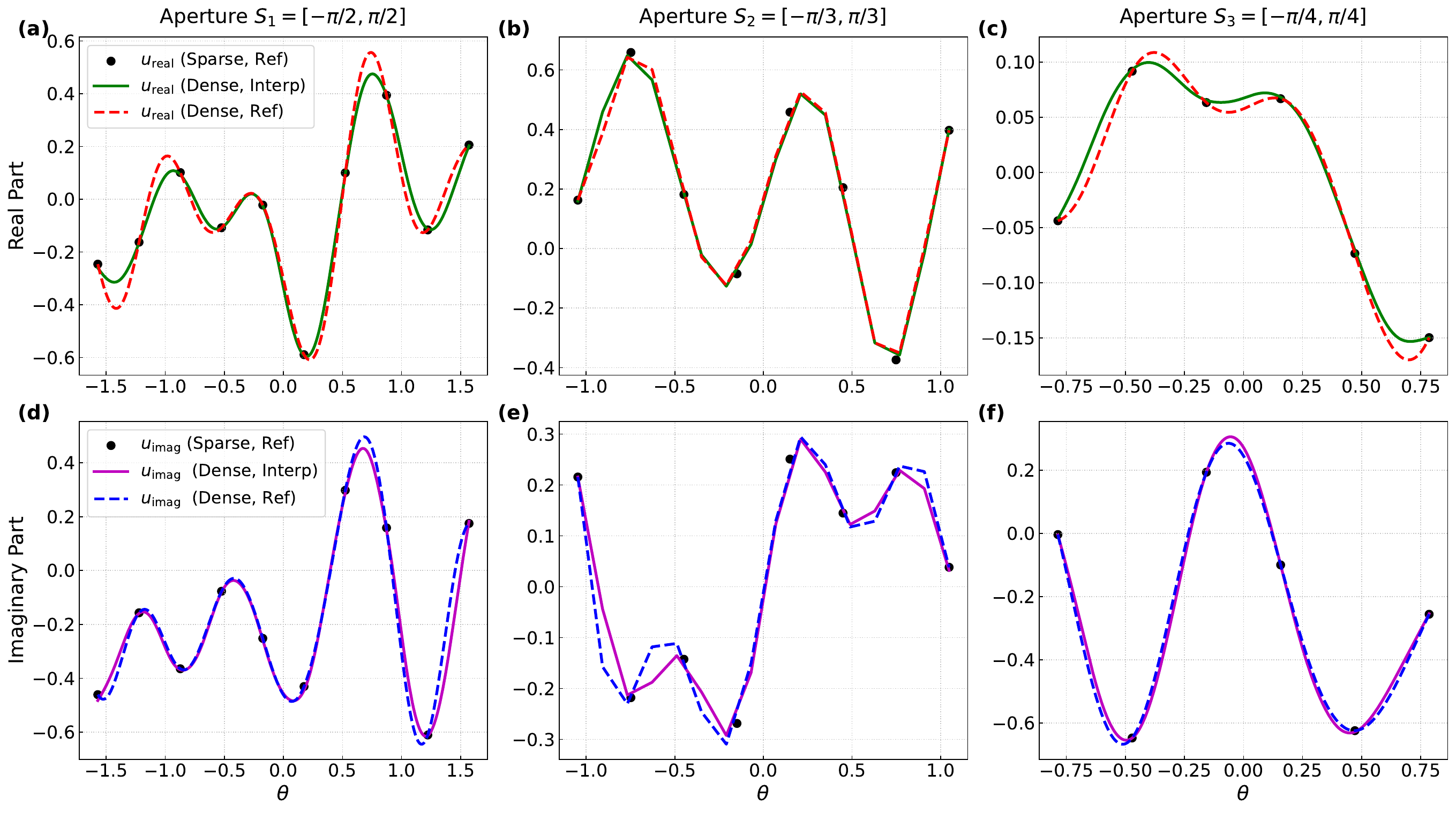}
    \caption{
        \textbf{Comparison of operator-interpolated dense traces with ground-truth data ($N=3$).} Real part (a-c) and imaginary part (d-f) of the predicted dense trace $\widehat u(\varphi)$ (solid), the reference dense trace $u(\varphi)$ (dashed), and the $M$ measurements (dots) for the
        three apertures.
    }
    \label{fig:compare_interp_measurement_3}
\end{figure}

\begin{figure}[htb]
    \centering
    \includegraphics[width=0.95\textwidth]{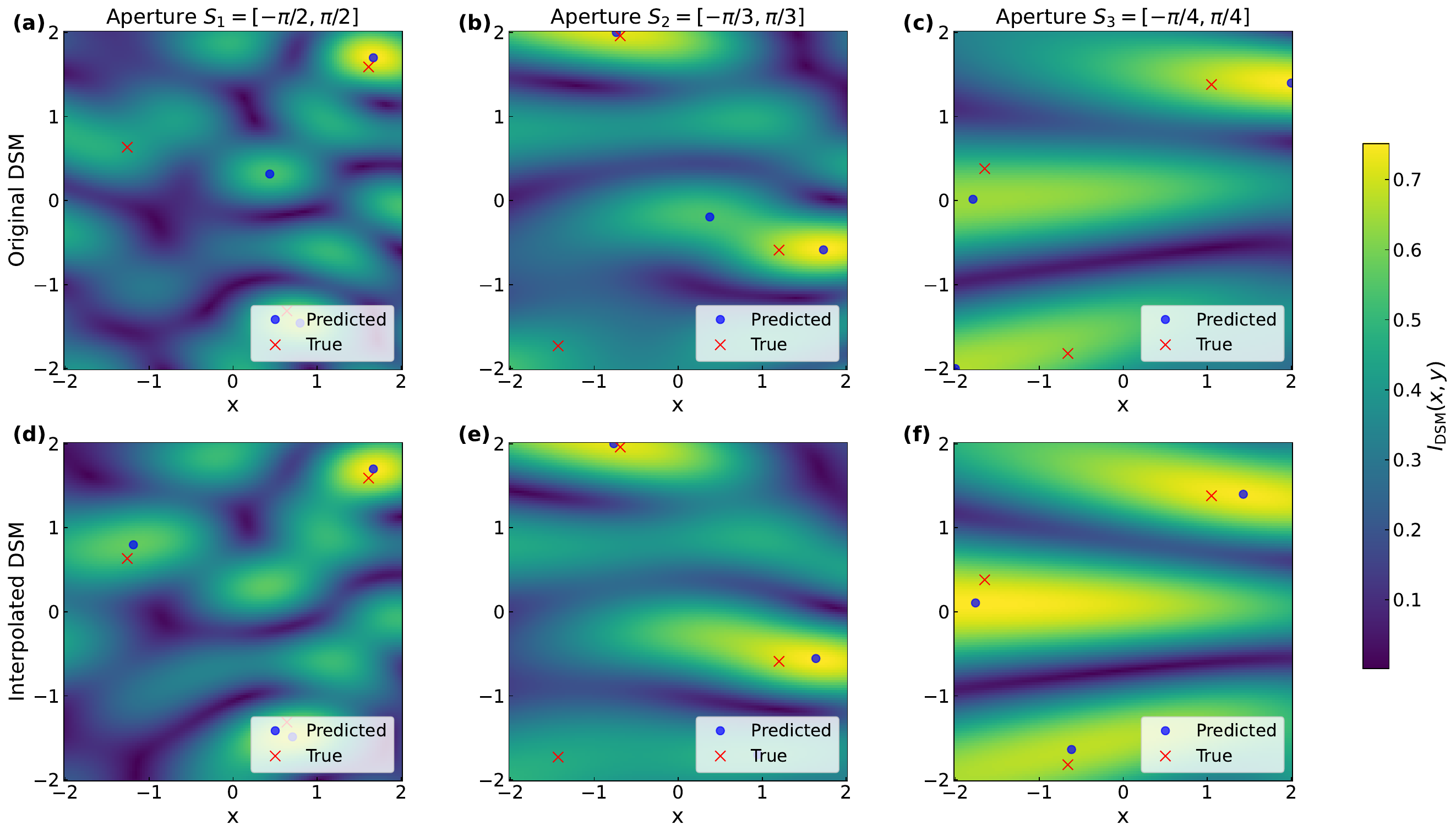}
    \caption{
        \textbf{Effect of operator interpolation on DSM source localization ($N=3$).} Top: DSM from raw sparse data; bottom: DSM from DeepONet-interpolated data; columns: apertures $[-\pi/n,\pi/n],n=2,3,4$. Red crosses: truth; blue circles: estimates. Mean absolute error drops from $0.573$, $0.728$, $0.547$ to $0.094$, $0.243$, $0.166$, respectively.
        }
    \label{fig:compare_dsm_interpolated_dsm_3}
\end{figure}

\begin{table}[htb]
  \caption{\textbf{Error comparison of DSM and DeepONet-interpolated DSM ($N=3$).}
  Mean absolute localization error (MAE) with DSM from raw sparse data and DeepONet-interpolated data (ours).}
  \label{tab:error_comparison_3}
  \centering
  \setlength{\tabcolsep}{4pt}
  \renewcommand{\arraystretch}{1.15}
  \resizebox{\textwidth}{!}{%
  \begin{tabular}{l*{3}{ccc}}
    \toprule
    & \multicolumn{3}{c}{$S_1$} & \multicolumn{3}{c}{$S_2$} & \multicolumn{3}{c}{$S_3$} \\
    \cmidrule(lr){2-4}\cmidrule(lr){5-7}\cmidrule(lr){8-10}
    Method & $z_{\text{true}}$ & $z_{\text{pred}}$ & MAE
           & $z_{\text{true}}$ & $z_{\text{pred}}$ & MAE
           & $z_{\text{true}}$ & $z_{\text{pred}}$ & MAE \\
    \midrule
    DSM
      & \begin{tabular}{@{}c@{}}$(1.61,\,1.59)$\\$(0.64,\,-1.31)$\\$(-1.26,\,0.63)$\end{tabular}
      & \begin{tabular}{@{}c@{}}$(1.669,\,1.669)$\\$(0.796,\,-1.458)$\\$(0.436,\,0.315)$\end{tabular}
      & 0.573
      & \begin{tabular}{@{}c@{}}$(1.20,\,-0.59)$\\$(-0.69,\,1.96)$\\$(-1.43,\,-1.73)$\end{tabular}
      & \begin{tabular}{@{}c@{}}$(1.729,\,-0.586)$\\$(-0.736,\,2.0)$\\$(0.375,\,-0.195)$\end{tabular}
      & 0.728
      & \begin{tabular}{@{}c@{}}$(1.05,\,1.38)$\\$(-0.66,\,-1.82)$\\$(-1.65,\,0.38)$\end{tabular}
      & \begin{tabular}{@{}c@{}}$(2.0,\,1.398)$\\$(-2.0,\,-2.0)$\\$(-1.789,\,0.015)$\end{tabular}
      & 0.547 \\
    \midrule
    Ours
      & \begin{tabular}{@{}c@{}}$(1.61,\,1.59)$\\$(0.64,\,-1.31)$\\$(-1.26,\,0.63)$\end{tabular}
      & \begin{tabular}{@{}c@{}}$(1.669,\,1.669)$\\$(0.706,\,-1.488)$\\$(-1.187,\,0.796)$\end{tabular}
      & 0.094
      & \begin{tabular}{@{}c@{}}$(1.20,\,-0.59)$\\$(-0.69,\,1.96)$\\$(-1.43,\,-1.73)$\end{tabular}
      & \begin{tabular}{@{}c@{}}$(1.458,\,-0.526)$\\$(-0.827,\,2.0)$\\$(-0.676,\,-1.729)$\end{tabular}
      & 0.243
      & \begin{tabular}{@{}c@{}}$(1.05,\,1.38)$\\$(-0.66,\,-1.82)$\\$(-1.65,\,0.38)$\end{tabular}
      & \begin{tabular}{@{}c@{}}$(1.428,\,1.398)$\\$(-0.616,\,-1.639)$\\$(-1.759,\,0.105)$\end{tabular}
      & 0.166 \\
    \bottomrule
  \end{tabular}}
\end{table}

A natural question is whether conventional interpolation can also provide dense data for the DSM method. 
In the following, we compare the DeepONet-based neural operator with three classical methods that use the same sparse sensor measurements as in Figure~\ref{fig:compare_other_interp_3}:
\begin{itemize}
    \item[i)] \textbf{Piecewise linear (PL).} Trend-preserving and non-oscillatory; however, it underfits curvature and exhibits systematic bias in regions of high curvature.
    \item[ii)] \textbf{Piecewise quadratic (PQ).} Better curvature fidelity with lower error on smooth segments; however, it is sensitive near extrema and shows enlarged endpoint errors, especially in the imaginary component.
    \item[iii)] \textbf{Global polynomial.} Interpolates the samples exactly and often fits well near the interval center; however, it suffers from Runge-type boundary oscillations, leading to large error spikes and poor robustness.
\end{itemize}
By contrast, the DeepONet surrogate attains among the lowest pointwise mean absolute errors and faithfully reconstructs oscillatory structures in both the real and imaginary parts. 
Trained offline without enforcing nodal exactness, DeepONet learns a smooth, low-variation mapping, yielding stable, low-variance reconstructions on unseen data and maintaining high global accuracy under severe sensor sparsity.

\begin{figure}[htb]
    \centering
    \includegraphics[width=0.95\textwidth]{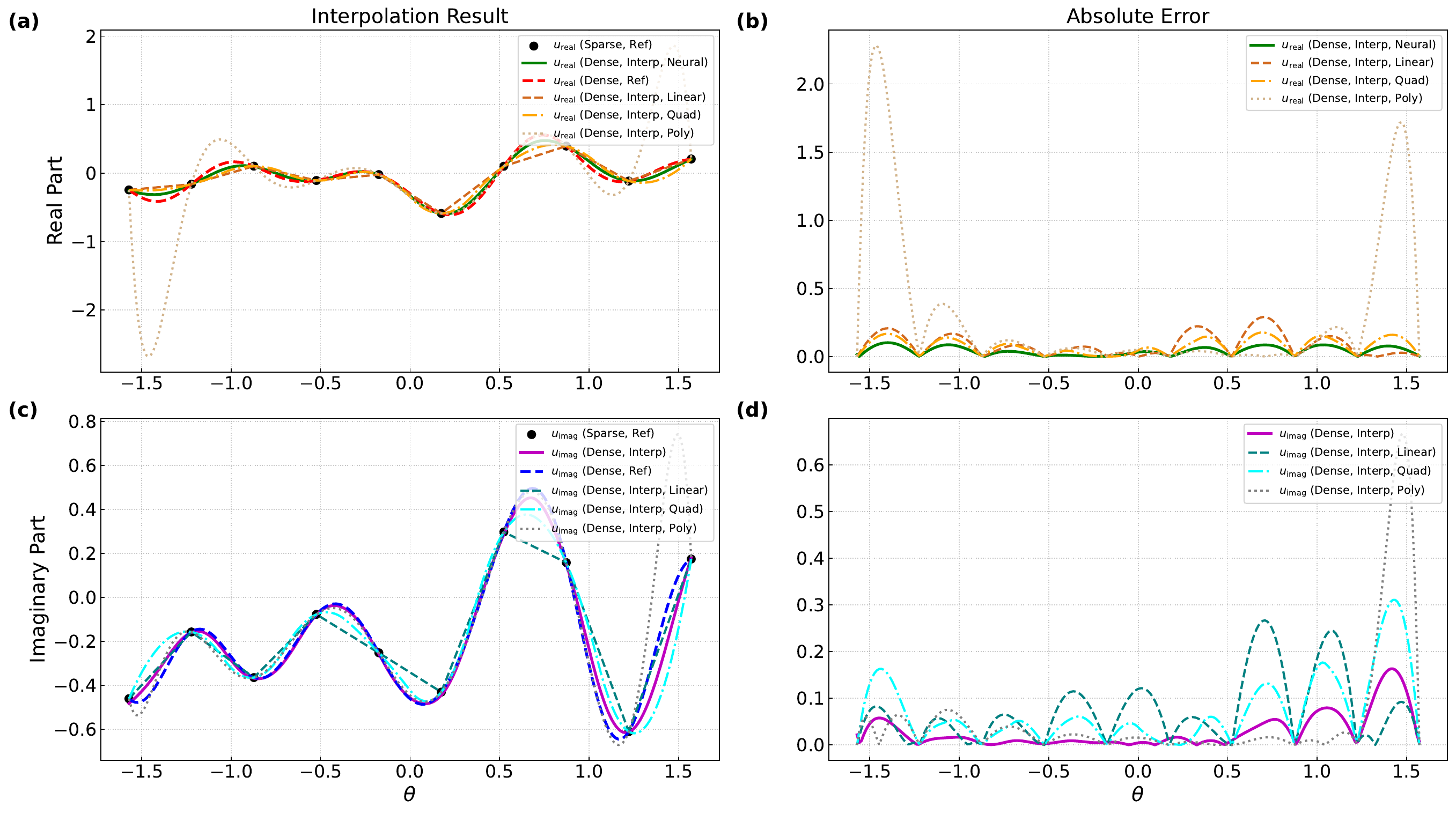}
    \caption{
        \textbf{Comparison of neural-operator interpolation with three classical methods on aperture $S_1=[-\pi/2,\pi/2]$.} 
        Panels (a,c): interpolated real and imaginary parts; (b,d): point-wise absolute errors.}
    \label{fig:compare_other_interp_3}
\end{figure}

We apply each interpolation scheme to densify the sparse measurements and feed the resulting fields into DSM. 
Localization performance is quantified by the mean absolute error in \eqref{eq:mae}. 
As shown in Figure~\ref{fig:compare_other_interp_mae}, the DeepONet-based interpolation has a clear advantage: it attains the lowest mean absolute error across resolutions, exhibits a sharp drop once the number of interpolation points exceeds $\approx 62$, and then remains low and stable. It captures fine oscillations without enforcing pointwise constraints and continues to benefit from additional sensors. 

By contrast, the piecewise-linear and piecewise-quadratic methods produce nearly flat MAE curves, adding points yields little to no improvement, which indicate limited expressivity and a persistent bias floor for oscillatory fields. The global polynomial interpolant suffers from Runge-type behavior, with large and erratic errors that do not reliably decrease as the sample count grows.

Using 128 points, Figure~\ref{fig:compare_other_interp_dsm} further confirms these trends: DSM with DeepONet delivers the highest-fidelity reconstructions and reliably localizes all three true point sources, whereas PL/PQ remain noticeably biased and the polynomial fit is unstable near the boundaries.

\begin{figure}[htb]
    \centering
    \includegraphics[width=0.55\textwidth]{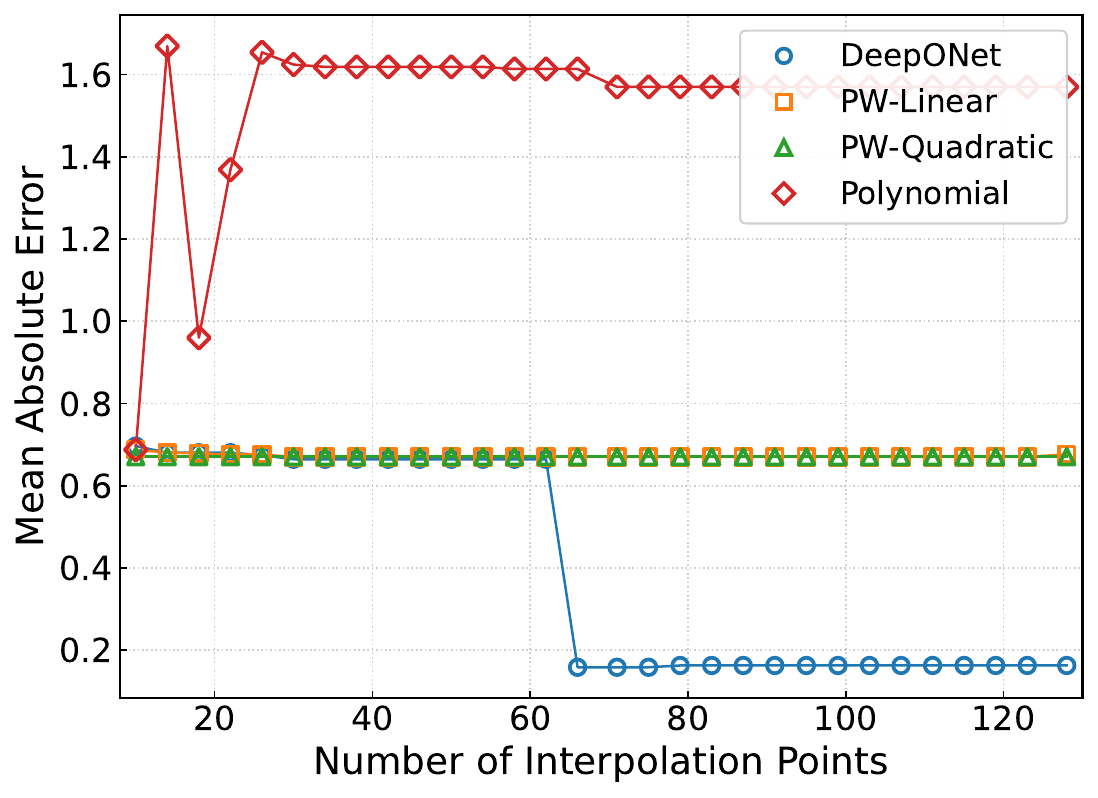}
    \caption{
        \textbf{Stability and accuracy of DSM with different interpolation schemes.} Mean absolute error versus the number of interpolation points $N$ used to densify the same sparse sensors.}
    \label{fig:compare_other_interp_mae}
\end{figure}

\begin{figure}[htb]
    \centering
    \includegraphics[width=0.65\textwidth]{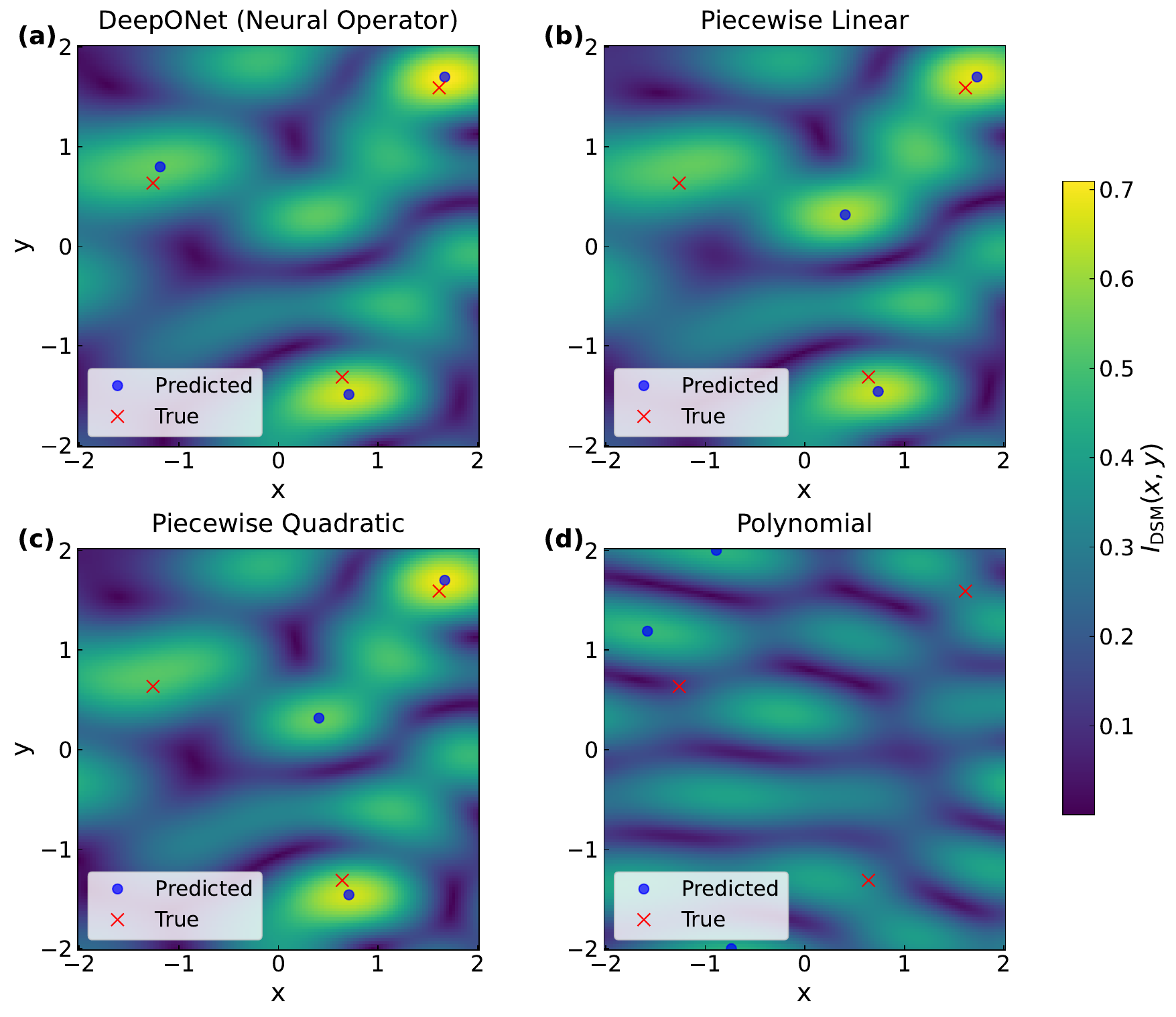}
    \caption{
        \textbf{DSM source localization with four interpolation schemes (128 points, $[-\pi/2, \pi/2]$).} Panels: (a) DeepONet (Neural Operator), (b) Piecewise Linear, (c) Piecewise Quadratic, (d) Polynomial.}
    \label{fig:compare_other_interp_dsm}
\end{figure}

\section{Discussion and Conclusion}\label{sec4}

In this work, we propose a modular framework for inverse Helmholtz source localization under limited measurements  acquired over a partial aperture. 
On the theory side, we extend the uniqueness theorem to finite aperture for Dirac-type sources, and show that the inverse Helmholtz problem involving Dirac-source terms  remains a uniquely solvable despite limited viewing aperture. We also derive an error estimate for the single-source case,  thereby indicating that grid-level accuracy is attainable even from sparse data.

Empirically, DSM deteriorates as the number of sources increases and under single-frequency measurements, due to limited aperture and sparse measurements. 
Notably, even with a fixed viewing aperture, densifying the measurements along the measurement curve improves DSM performance. 
Classical interpolation methods are sensitive to the sparsity pattern and often underperform on such sparse data, whereas data-driven operator learning can fit a reconstruction map from many sparse-to-dense exemplars.
We therefore learn a measurement-reconstruction operator using a  DeepONet with a branch–trunk architecture, trained offline to interpolate as few as six to ten measurements into a dense, self-consistent synthetic aperture. The interpolated data is then passed to DSM. 

The modularity of both completion and inversion components enables plug-and-play integration of  neural-operator variants and advanced inversion methods. 
We anticipate that the same paradigm will benefit a broad class of limited-aperture imaging modalities, including underwater acoustics, medical ultrasound, and through-wall radar.

\section*{Acknowledgments}

This work was supported by the National Key R \& D Program of China (2024YFA1012600) and NSFC Project (12431014).



\begin{thebibliography}{99}
\bibitem{ahmad2008multi} F. Ahmad and M. G. Amin,
    \textit{Multi-location wideband synthetic aperture imaging for urban sensing applications},
    Journal of the Franklin Institute,
    345(6)(2008),
    618-639.
\bibitem{anastasio2007application} M. A. Anastasio,
    J. Zhang,
    D. Modgil,
    and P. J. La Rivi{\`e}re,
    \textit{Application of inverse source concepts to photoacoustic tomography},
    Inverse Problems,
    23(6)(2007),
    S21.
\bibitem{arridge1999optical} S. R. Arridge,
    \textit{Optical tomography in medical imaging},
    Inverse Problems,
    15(2)(1999),
    R41.
\bibitem{bao2002inverse} G. Bao,
    H. Ammari,
    and J. L. Fleming,
    \textit{An inverse source problem for Maxwell's equations in magnetoencephalography},
    SIAM Journal on Applied Mathematics,
    62(4)(2002),
    1369-1382.
\bibitem{bao2010multi} G. Bao,
    J. Lin,
    and F. Triki,
    \textit{A multi-frequency inverse source problem},
    Journal of Differential Equations,
    249(12)(2010),
    3443-3465.
\bibitem{bao2015inverse} G. Bao,
    P. Li,
    J. Lin,
    and F. Triki,
    \textit{Inverse scattering problems with multi-frequencies},
    Inverse Problems,
    31(9)(2015),
    093001.
\bibitem{bleistein1977nonuniqueness} N. Bleistein and J. K. Cohen,
    \textit{Nonuniqueness in the inverse source problem in acoustics and electromagnetics},
    Journal of Mathematical Physics,
    18(2)(1977),
    194-201.
\bibitem{colton1996simple} D. Colton and A. Kirsch,
    \textit{A simple method for solving inverse scattering problems in the resonance region},
    Inverse Problems,
    12(4)(1996),
    383.
\bibitem{D.CandR.K} D. Colton and R. Kress,
    \textit{Inverse Acoustic and Electromagnetic Scattering Theory},
    3rd ed., Springer,
    New York,
    (2013).
\bibitem{cybenko1989approximation} G. Cybenko,
    \textit{Approximation by superpositions of a sigmoidal function},
    Mathematics of Control,
    Signals and Systems,
    2(4)(1989),
    303-314.
\bibitem{devaney2003nonuniqueness} A. Devaney and G. Sherman,
    \textit{Nonuniqueness in inverse source and scattering problems},
    IEEE Transactions on Antennas and Propagation,
    30(5)(2003),
    1034-1037.
\bibitem{el2011inverse} A. El Badia and T. Nara,
    \textit{An inverse source problem for Helmholtz's equation from the Cauchy data with a single wave number},
    Inverse Problems,
    27(10)(2011),
    105001.
\bibitem{hohage1998convergence} T. Hohage,
    \textit{Convergence rates of a regularized Newton method in sound-hard inverse scattering},
    SIAM Journal on Numerical Analysis,
    36(1)(1998),
    125-142.
\bibitem{LHormander} L. Hormander,
    \textit{The Analysis of Linear Partial Differential Operators,
    I},
    Springer-Verlag, New York,
    (1983).
\bibitem{ito2012direct} K. Ito,
    B. Jin,
    and J. Zou,
    \textit{A direct sampling method to an inverse medium scattering problem},
    Inverse Problems,
    28(2)(2012),
    025003.
\bibitem{jensen2011computational} F. B. Jensen,
    W. A. Kuperman,
    M. B. Porter,
    H. Schmidt,
    and A. Tolstoy,
    \textit{Computational Ocean Acoustics},
    Springer,
    (2011).
\bibitem{ji2020identification} X. Ji and X. Liu,
    \textit{Identification of Point-Like Objects with Multifrequency Sparse Data},
    SIAM Journal on Scientific Computing,
    42(4)(2020),
    A2325-A2343.
\bibitem{ji2021source} X. Ji and X. Liu,
    \textit{Source Reconstruction with Multifrequency Sparse Scattered Fields},
    SIAM Journal on Applied Mathematics,
    81(6)(2021),
    2387-2404.
\bibitem{kirsch1987optimization} A. Kirsch and R. Kress,
    \textit{An optimization method in inverse acoustic scattering},
    Boundary elements IX,
    3(1987), 3-18.
\bibitem{kirsch1998characterization} A. Kirsch,
    \textit{Characterization of the shape of a scattering obstacle using the spectral data of the far field operator},
    Inverse Problems,
    14(6)(1998),
    1489.
\bibitem{kovachki2023neural} N. Kovachki,
    Z. Li,
    B. Liu,
    K. Azizzadenesheli,
    K. Bhattacharya,
    A. Stuart,
    and A. Anandkumar,
    \textit{Neural operator: Learning maps between function spaces with applications to pdes},
    Journal of Machine Learning Research,
    24(89)(2023),
    1-97.
\bibitem{li2020extended} Z. Li,
    Z. Deng,
    and J. Sun,
    \textit{Extended-sampling-Bayesian method for limited aperture inverse scattering problems},
    SIAM Journal on Imaging Sciences,
    13(1)(2020),
    422-444.
\bibitem{li2021quality} Z. Li,
    Y. Liu,
    J. Sun,
    and L. Xu,
    \textit{Quality-Bayesian approach to inverse acoustic source problems with partial data},
    SIAM Journal on Scientific Computing,
    43(2)(2021),
    A1062-A1080.
\bibitem{liu2022deterministic} Y. Liu,
    Z. Wu,
    J. Sun,
    and Z. Zhang,
    \textit{Deterministic-statistical approach for an inverse acoustic source problem using multiple frequency limited aperture data},
    Inverse Problems and Imaging,
    17(6)(2023), 
    1329-1345.
\bibitem{loshchilov2016sgdr} I. Loshchilov and F. Hutter,
    \textit{SGDR: Stochastic Gradient Descent with Warm Restarts},
    International Conference on Learning Representations,
    (2017).
\bibitem{loshchilov2017decoupled} I. Loshchilov and F. Hutter,
    \textit{Decoupled Weight Decay Regularization},
    International Conference on Learning Representations,
    (2019).
\bibitem{lu2021learning} L. Lu,
    P. Jin,
    G. Pang,
    Z. Zhang,
    and G. E. Karniadakis,
    \textit{Learning nonlinear operators via DeepONet based on the universal approximation theorem of operators},
    Nature Machine Intelligence,
    3(3)(2021),
    218-229.
\bibitem{Qin_2023} X. Qin and Y. Li,
    \textit{Based on Gradient Algorithm for the Inverse Source Problem of a Class of Time-space Fractional Diffusion Equations},
    Journal of Physics: Conference Series,
    2650(1)(2023),
    012003.
\bibitem{BesselFunctions} G. N.Watson,
    \textit{A Treatise on the Theory of Bessel Functions}, 
    Cambridge University Press, 
    (1996).
\bibitem{YangZhangBoandZhang} J. Yang,
    B. Zhang,
    and H. Zhang,
    \textit{Reconstruction of Complex Obstacles with Generalized Impedance Boundary Conditions from Far-Field Data},
    SIAM Journal on Applied Mathematics,
    74(1)(2014),
    106-124.
\bibitem{zhang2018locating} D. Zhang,
    Y. Guo,
    J. Li,
    and H. Liu,
    \textit{Locating multiple multipolar acoustic sources using the direct sampling method},
    Communications in Computational Physics,
    25(5)(2019),
    1328-1356.
\end{thebibliography}


\appendix

\section{Estimation of $g'(0)$ and $g'(\frac{1}{15k})$}\label{apd:1}

(i) For $k>0,\lambda>0$, we have $g'(0)>0$.
 
\begin{proof}
Substituting \(y = 0\) into \eqref{Dg}, we can obtain    
\begin{align}
    g'\left( 0 \right) =\lambda\frac{G(k\xi)}{|H_{0}^{1}\left( k\xi \right) |^3},
\end{align}
where $G(x):=J_0\left( x \right) J_1\left( x\right) +Y_0\left( x \right) Y_1\left( x\right)$.

We  need to prove that \( G(x) > 0 \), $\forall  x > 0 $. 
From \cite[p.444(1)]{BesselFunctions}, we have
\begin{align}\label{sqaure}
    J_{\nu}\left( x \right) ^2+Y_{\nu}\left( x \right) ^2=\frac{8}{\pi ^2}\int_0^{\infty}{K_0\left( 2x\sinh t \right) \cosh 2\nu t}\mathrm{d}t,    
\end{align}
where $K_0$ is the modified Bessel function of the second kind of zero order, and  in
general,  its integral expression is
\begin{align*}
    K_{\nu}\left( x \right) =\int_0^{\infty}{e^{-x\cosh u}}\cosh \nu u\mathrm{d}u,
\end{align*}
Substituting \(\nu = 0\) into \eqref{sqaure}, then differentiate both sides with respect to $x$; it follows that
\begin{align*} 
    &-2J_0\left( x \right) J_1\left( x \right) -2Y_0\left( x \right) Y_1\left( x \right) =\frac{8}{\pi ^2}\frac{\mathrm{d}}{\mathrm{d}x}\int_0^{\infty}{K_0\left( 2x\sinh t \right)}\mathrm{d}t \\ 
    &=\frac{8}{\pi ^2}\int_0^{\infty}{\int_0^{\infty}{\frac{\mathrm{d}}{\mathrm{d}x}e^{-2x\sinh t\cosh u}}\mathrm{d}u}\mathrm{d}t \\ 
    &=-\frac{16}{\pi ^2}\int_0^{\infty}{\sinh t\int_0^{\infty}{e^{-2x\sinh t\cosh u}}\cosh u\mathrm{d}u}\mathrm{d}t \\ 
    &=-\frac{16}{\pi ^2}\int_0^{\infty}{K_1\left( 2x\sinh t \right) \sinh t}\mathrm{d}t. 
\end{align*}

Namely 
\begin{align*}
    J_0\left( x \right) J_1\left( x \right) +Y_0\left( x \right) Y_1\left( x \right) &=\frac{8}{\pi ^2}\int_0^{\infty}{K_1\left( 2x\sinh t \right) \sinh t}\mathrm{d}t
    \\
    &=\frac{8}{\pi ^2}\int_0^{\infty}{\int_0^{\infty}{e^{-2x\sinh t\cosh u}}\sinh t}\cosh u\mathrm{d}t\mathrm{d}u>0.
\end{align*}
Thus, we have proved that  $G(x) > 0 ,x > 0,$ and then $g'(0)>0.$
\end{proof}

(ii) For $k\xi\ge 15 ,\lambda>0$, we have
 $g'(\frac{1}{15k})<0.$
 
\begin{proof}
    Substituting \(y = \frac{1}{15k}\) into \eqref{Dg}, we can obtain 
\begin{align}
    g'\left( \frac{1}{15k} \right) =-\lambda J_1(\frac{1}{15})\frac{1}{K(\frac{1}{15}+k\xi)}+\lambda J_0(\frac{1}{15})\frac{G\left( \frac{1}{15}+k\xi \right)}{K(\frac{1}{15}+k\xi)^3},
\end{align}
where $G(x):=J_0\left( x\right) J_1\left(x \right) +Y_0\left(  x\right) Y_1\left( x \right),K(x):=|H_{0}^{1}\left( x \right) |$.

Setting \( z = k\xi \), we define \( A(z) \) as:  
\begin{align*}
    A(z) = -\lambda J_1\left( \frac{1}{15} \right)\frac{1}{K\left( z + \frac{1}{15} \right)} + \lambda J_0\left( \frac{1}{15} \right)\frac{G\left( z + \frac{1}{15} \right)}{K\left( z + \frac{1}{15} \right)^3},
\end{align*}
since  \( A'(z) < 0 \) holds for \( z \in (0, \infty) \), it follows that  \( A(z) \)  is strictly  decreasing on \( (0, \infty) \).  

Thus, for any  \( z \in [15, \infty) \), we have:  
\[
g'\left( \frac{1}{15k} \right) = A(z) \le A(15) = -0.000982197 < 0,
\]  
the proof is completed.
\end{proof}

\end{document}